\documentclass{amsart}

\usepackage{amssymb, amsmath, esint}
\usepackage{hyperref}

\renewcommand{\[}{\begin{equation}\begin{aligned}}
\renewcommand{\]}{\end{aligned} \end{equation}}

\newcommand{\ddbar}{\sqrt{-1}\partial \bar\partial }

\usepackage{verbatim}

\newtheorem{thm}{Theorem}
\newtheorem{prop}[thm]{Proposition}
\newtheorem{lemma}[thm]{Lemma}

\newtheorem{cor}[thm]{Corollary}
\newtheorem{conj}[thm]{Conjecture}

\newcommand{\tRic}{\widetilde{\mathrm{Ric}}}

\newcommand{\ssubset}{\subset\joinrel\subset}

\theoremstyle{remark}
\newtheorem{remark}[thm]{Remark}

\theoremstyle{definition}
\newtheorem{definition}[thm]{Definition}

\author{G\'abor Sz\'ekelyhidi}
\address{Department of Mathematics, Northwestern University, Evanston,
  IL, USA}
\email{gaborsz@northwestern.edu}

\title{Singular K\"ahler-Einstein metrics and RCD spaces}
\date{}

\begin{document}

\begin{abstract}
  We study K\"ahler-Einstein metrics on singular projective
  varieties. We show that under an approximation property with
  constant scalar curvature metrics, the metric completion of the
  smooth part is a non-collapsed RCD space, and is homeomorphic to
  the original variety. 
\end{abstract}

\maketitle

\section{Introduction}
A basic idea in complex geometry is to study complex manifolds using
canonical K\"ahler metrics, of which perhaps the most important
examples are K\"ahler-Einstein
metrics. Yau's solution of the Calabi conjecture~\cite{Yau78} provides
K\"ahler-Einstein metrics on compact K\"ahler manifolds with negative
or zero first Chern class, while Chen-Donaldson-Sun's solution of the
Yau-Tian-Donaldson conjecture~\cite{CDS3} shows that a Fano manifold
admits a K\"ahler-Einstein metric if and only if it is K-stable. An
example of a geometric application of such metrics is Yau's
proof~\cite{Yau77} of the Miyaoka-Yau inequality.

Recently there has been increasing interest in K\"ahler-Einstein
metrics on singular varieties. In particular Yau's theorem was
extended to the singular case by
Eyssidieux-Guedj-Zeriahi~\cite{EGZ}, while the singular case of the
Yau-Tian-Donaldson conjecture was finally resolved by
Liu-Xu-Zhuang~\cite{LXZ22} after many partial results (see for instance
\cite{LTW21}). There is now a substantial literature on
singular K\"ahler-Einstein metrics, see e.g. \cite{BBEGZ,Ber16, GGK19,LTW21, GPSS2}. 

In order to state our main results, suppose that $X$ is an
$n$-dimensional normal compact K\"ahler
space. Let us recall that a singular
K\"ahler-Einstein metric on $X$ can be defined to be a positive current
$\omega_{KE}$ that is a smooth K\"ahler metric on the regular set $X^{reg}$, has locally
bounded potentials, and satisfies the equation $\mathrm{Ric}(\omega_{KE}) =
\lambda \omega_{KE}$ on $X^{reg}$ for a constant $\lambda\in
\mathbb{R}$. The metric $\omega_{KE}$ defines a length metric $d_{KE}$
on $X^{reg}$, and an important problem is to understand the geometry of
the metric completion $\overline{(X^{reg}, d_{KE})}$.

In recent remarkable works,
Guo-Phong-Song-Sturm~\cite{GPSS1,GPSS2} showed that this metric
completion satisfies many important geometric estimates, such as bounds for their
diameters, their heat kernels, as well as Sobolev inequalities, even
under far more general assumptions than the Einstein condition. In
particular, their results do not assume Ricci curvature bounds. It is
natural to expect, however, that singular K\"ahler-Einstein metrics satisfy sharper
results, similar to Riemannian manifolds with Ricci lower bounds. We formulate the
following conjecture, which is likely folklore among experts, although
we did not find it stated in the literature in this generality.

\begin{conj}\label{conj:RCD}
  The metric completion $\overline{(X^{reg}, d_{KE})}$, equipped with
  the measure $\omega_{KE}^n$, extended trivially from $X^{reg}$, is a
  non-collapsed $RCD(\lambda, 2n)$-space, homeomorphic to $X$.
\end{conj}

The notion of non-collapsed RCD-space is due to De
Philippis-Gigli~\cite{DPG}, building on many previous works on
synthetic notions of Ricci curvature lower bounds (see
\cite{SturmKT,LV09,AGS}). 
The conjecture is already known
in several special cases, where in fact $\overline{(X^{reg}, d_{KE})}$
is shown to be a non-collapsed Ricci limit space -- these are
non-collapsed Gromov-Hausdorff limits of Riemannian manifolds with
Ricci lower bounds, studied by Cheeger-Colding~\cite{CC1}. Settings
where $\overline{(X^{reg}, d_{KE})}$ is a Ricci limit space are given,
for example, 
by K-stable Fano manifolds with admissible singularities (see
Li-Tian-Wang~\cite{LTW21}, or Song~\cite{Song14} for the case of
crepant singularities), or smoothable K-stable Fano varieties, see
Donaldson-Sun~\cite{DS1}, Spotti~\cite{Spotti}. 

Our goal in this paper is to move beyond the setting of Ricci limit
spaces, and to prove the conjecture in situations where it is not
clear whether the singular
K\"ahler-Einstein space $(X, \omega_{KE})$ can be approximated by
smooth, or mildly singular, spaces with Ricci curvature bounded below. Instead, our
approach is to use an approximation with constant scalar curvature
K\"ahler metrics. The main approximation property that we
require is the following.

\begin{definition}\label{defn:cscKapprox}
  We say that the singular K\"ahler-Einstein space $(X, \omega_{KE})$
  can be
  \emph{approximated by cscK metrics}, if there
  is a resolution $\pi: Y\to X$, and a family of constant scalar
  curvature K\"ahler metrics $\omega_\epsilon$ on $Y$ satisfying the following:
  \begin{enumerate}
    \item[(a)] We have $\omega_\epsilon = \eta_\epsilon + \ddbar
      u_\epsilon$, where $\eta_\epsilon$ converge smoothly to
      $\pi^*\eta_X$ and $\eta_\epsilon \geq \pi^*\eta_X$.
      Here $\eta_X\in [\omega_{KE}]$ is a smooth metric
      on $X$ in the sense that it is locally the restriction of smooth
      metrics under local embeddings into Euclidean space.  
    \item[(b)] We have the estimates
      \[ \sup_Y |u_\epsilon| < C, \quad
        \frac{\omega_\epsilon^n}{\eta_Y^n} > \gamma, \quad \int_Y 
        \left(\frac{\omega_\epsilon^n}{\eta_Y^n}\right)^p\
        \eta_Y^n < C, \]
      for constants $C > 0, p > 1$ independent of $\epsilon$, where $\eta_Y$ is
      a fixed K\"ahler metric on $Y$, and $\gamma$ is a non-negative
      continuous function on $Y$ vanishing only along the exceptional
      divisor, also independent of $\epsilon$. 
    \item[(c)] The metrics $\omega_\epsilon$ converge locally smoothly on
      $\pi^{-1}(X^{reg})$ to $\pi^*\omega_{KE}$. 
  \end{enumerate}
\end{definition}

The cscK property of the approximating metrics $\omega_\epsilon$ is
used to obtain integral bounds for the Ricci and Riemannian curvatures
as in Proposition~\ref{prop:cscKbounds}. We expect that such an approximation is possible in
all cases of interest, however at the moment this is only known in
limited settings. We have the following result.

\begin{thm}\label{thm:cscKapprox}
  Suppose that $(X, \omega_{KE})$ is a singular K\"ahler-Einstein
  space with $\omega_{KE}\in c_1(L)$ for a line bundle over $X$, and
  such that $X$ has discrete automorphism group. Assume that $X$
  admits a projective resolution $\pi: Y\to X$ for which the 
  anticanonical bundle $-K_Y$ is relatively nef over $X$. Then $(X, \omega_{KE})$
  can be  approximated by cscK metrics in the sense of the definition
  above. 
\end{thm}

Note that recently Boucksom-Jonsson-Trusiani~\cite{BJT24} showed the 
existence of cscK metrics on resolutions in this setting (and even
more generally), while Pan-T{\^o}~\cite{PT24} showed estimates for
these approximating cscK metrics closely related to those in
Definition~\ref{defn:cscKapprox}, in a more general setting. 

Our main result on K\"ahler-Einstein spaces that can be approximated
by cscK metrics is the following.
\begin{thm}\label{thm:main}
  Suppose that $(X, \omega_{KE})$ can be approximated by cscK metrics,
  and $\omega_{KE}\in
  c_1(L)$ for a line bundle $L$ on $X$. Then Conjecture~\ref{conj:RCD} holds for
  $\overline{(X^{reg}, d_{KE})}$. In addition the metric singular set of
  $\overline{(X^{reg}, d_{KE})}$ agrees with the complex
  analytic singular set $X\setminus X^{reg}$, and it has Hausdorff
  codimension at least 4. 
\end{thm}

It is natural to expect that Conjecture~\ref{conj:RCD} can also be
extended to the setting of singular K\"ahler-Einstein metrics $\omega$
with cone singularities along a divisor on klt pairs $(X,D)$. In this case
one can hope to approximate $\omega$ using cscK metrics with cone
singularities on a log resolution of $(X,D)$. Some results in this direction
were obtained recently by Zheng~\cite{Zheng}, but we leave this
extension of Theorem~\ref{thm:main} for future
work.

The RCD property
implies important geometric information about the metric completion
$\overline{(X^{reg}, d_{KE})}$, such as the existence of tangent
cones (see De Philippis-Gigli~\cite{DPG2}).
Moreover, we expect that with only minor modifications the work of
Donaldson-Sun~\cite{DS2} and Li-Wang-Xu~\cite{LWX} on the tangent cones of smoothable
K\"ahler-Einstein spaces can be extended to the setting of
Theorem~\ref{thm:main}, i.e. the tangent cones of $\overline{(X^{reg},
  d_{KE})}$ are unique, and are determined by the algebraic
structure. Knowledge of the tangent cones can then be further
leveraged to obtain more refined information about the metric, such as
in Hein-Sun~\cite{HS}, or \cite{CSz23}.

Using results of Honda~\cite{Honda}, which rely on different equivalent
characterizations of RCD spaces by
Ambrosio-Gigli-Savar\'e~\cite{AGS}, the main estimate
that we need in order to prove the RCD property in
Theorem~\ref{thm:main}
is that eigenfunctions of the Laplacian are
Lipschitz continuous on $(X^{reg}, d_{KE})$. We will review Honda's
result in Section~\ref{sec:background}. In order to prove a gradient
estimate for eigenfunctions, we use the approximating smooth cscK spaces $(Y,
\omega_\epsilon)$. Note that these do not satisfy uniform gradient
estimates, since they do not have uniform Ricci curvature bounds from below.
Instead we will prove a weaker
estimate on $(Y, \omega_\epsilon)$, expressed in terms of the heat
flow -- roughly speaking we obtain an estimate that is valid for times
$t > t_\epsilon > 0$ along the heat flow, where $t_\epsilon\to 0$ as
$\epsilon\to 0$. These estimates can be passed to the limit as
$\epsilon \to 0$ using the uniform estimates of
Guo-Phong-Song-Sturm~\cite{GPSS1, GPSS2} for
the heat kernels, and in the limit we obtain the required gradient
bound on $(X^{reg}, \omega_{KE})$. This is discussed in
Section~\ref{sec:singularKERCD}. 

In Section~\ref{sec:DS} we prove that $\overline{(X^{reg}, d_{KE})}$
is homeomorphic to $X$, and that the metric singular set has Hausdorff codimension at
least 4. Some results of this type
were shown by Song~\cite{Song14} and La
Nave-Tian-Zhang~\cite{LNTZ}, based on applying H\"ormander's
$L^2$-estimates, following Donaldson-Sun~\cite{DS1}.
The main new
difficulty in our setting is that a priori we do not have enough control
of how large the set $\overline{(X^{reg}, d_{KE})} \setminus X^{reg}$ is
in the metric sense. It was shown by Sturm~\cite{SturmPrivate} (see also
\cite{Song14}), that this set has capacity zero, which already plays an
important role in the RCD property. For the
approach of Donaldson-Sun~\cite{DS1} to apply, however, we need a slightly
stronger effective bound that can be applied uniformly at all scales. In
previous related results this type of estimate relied on showing
that the metric regular set in $\overline{(X^{reg}, d_{KE})}$
coincides with $X^{reg}$, but this is not clear in our
setting since our approximating Riemannian manifolds $(Y, \omega_\epsilon)$ do not have lower
Ricci bounds. 

The new ingredient that we exploit is that the algebraic singular set of $X$ is
locally cut out by holomorphic (and therefore harmonic) functions. We
show that these functions have finite order of vanishing along the
singular set, and therefore we can control the size of their zero
sets in any ball that is sufficiently close to a Euclidean ball, using
a three annulus lemma argument, 
somewhat similarly to \cite{CJ20}. This leads to the key result that
the metric and algebraic regular sets of $\overline{(X^{reg},
  d_{KE})}$ coincide. After this the proof
follows by now familiar lines from Donaldson-Sun~\cite{DS1} and other subsequent
works such as \cite{LSz1}. 

In Section~\ref{sec:cscKapprox} we prove
Theorem~\ref{thm:cscKapprox}. The proof is based primarily on
Chen-Cheng's existence theorem for cscK metrics~\cite{ChenCheng} together
with some extensions of their estimates by Zheng~\cite{Zheng20}.  A
similar result, in more general settings, was obtained recently by
Boucksom-Jonsson-Trusiani~\cite{BJT24} and
Pan-T\^o~\cite{PT24}. 

In Section~\ref{sec:bounded}, as an example application,
we discuss an extension of Donaldson-Sun's partial $C^0$-estimate
 to singular K\"ahler-Einstein spaces
with the cscK approximation property. An additional ingredient that we
need is the gap result for the volume densities of (singular) Ricci flat K\"ahler
cone metrics that arise as tangent cones,
Theorem~\ref{thm:gap}. This was shown very recently in the more general
algebraic setting by Xu-Zhuang~\cite{XZ24}.

\subsection*{Acknowledgements}
I would like to thank Aaron Naber, Max Hallgren, Yuchen Liu, Tam\'as
Darvas, Valentino Tosatti, Jian Song, Yuji Odaka, Mattias Jonsson,
Sebastien Boucksom, and Antonio Trusiani for helpful discussions. In
addition I'm grateful to Chung-Ming Pan and Tat Dat T\^o for sharing
their preprint~\cite{PT24}. This work was supported in 
part by NSF grant DMS-2203218.

\section{Background}\label{sec:background}

\subsection{Non-collapsed RCD spaces}
By a metric measure space we mean a triple $(Z, d, \mathfrak{m})$,
where $(Z,d)$ is a metric space, and $\mathfrak{m}$ is a measure
on $Z$ with $\mathrm{supp}\,\mathfrak{m}=Z$. By now there are several
different, but essentially equivalent, notions of synthetic lower
bounds for the Ricci curvature of $(Z,d,\mathfrak{m})$, due to
Sturm~\cite{SturmKT}, Lott-Villani~\cite{LV09} and
Ambrosio-Gigli-Savar\'e~\cite{AGS}. We will be particularly
concerned with the notion of non-collapsed RCD($K,N$) space introduced
by De Philippis-Gigli~\cite{DPG}. These should be thought of as the
synthetic version of non-collapsed Gromov-Hausdorff limits of
$N$-dimensional manifolds with Ricci curvature bounded below by $K$.

More specifically we will be concerned with RCD spaces that are the metric
completions of smooth Riemannian manifolds. In fact the spaces that we
study almost fit into the setting of  {\em almost smooth metric measure
  spaces}, studied
by Honda~\cite{Honda}, except we will use the standard notion of
zero capacity set instead of \cite[Definition 3.1, 3(b)]{Honda}. The
results of \cite{Honda} hold with this definition too, as we will
outline below. Thus we state the following slight modification of
Honda's definition. 

\begin{definition} \label{defn:almostsmooth}
  A compact metric measure space $(Z, d, \mathfrak{m})$ is an
  $n$-dimensional almost smooth metric measure space, if there is an
  open subset $\Omega\subset Z$ satisfying the following conditions.
  \begin{itemize}
    \item[(1)] There is a smooth $n$-dimensional Riemannian manifold
      $(M,g)$ and a homeomorphism $\phi: \Omega \to M^n$, such that
      $\phi$ defines a local isometry between $(\Omega, d)$ and $(M^n,
      d_g)$.
    \item[(2)] The restriction of the measure $\mathfrak{m}$ to
      $\Omega$ coincides with the $n$-dimensional Hausdorff measure.
    \item[(3)] The complement $Z\setminus \Omega$ has measure zero,
      i.e. $\mathfrak{m}(Z\setminus \Omega)=0$, and it has zero
      capacity in the following sense: there is a sequence of smooth functions
      $\phi_i: \Omega\to [0,1]$ with compact support in $\Omega$ such that
      \begin{itemize}
        \item[(a)] For any compact $A\subset \Omega$ we have
          $\phi_i|_A=1$ for sufficiently large $i$,
        \item[(b)] We have
          \[ \lim_{i\to\infty} \int_\Omega |\nabla \phi_i|^2\,
            d\mathcal{H}^n = 0. \]
        \end{itemize}
    \end{itemize}
  \end{definition}

As a point of comparison we remark that in \cite{Honda}, the
condition (b) is replaced by requiring that the
$L^1$-norm of $\Delta \phi_i$ is uniformly bounded. Note that neither
of these conditions implies the other one. 

In our setting we will have an $n$-dimensional normal
projective variety $X$
equipped with a positive current $\omega$ that is a smooth K\"ahler
metric on
$X^{reg}$. In addition we will assume that $\omega$ has
locally bounded K\"ahler potentials. We use $\omega$ to define a
metric structure $d$ on the smooth locus
$X^{reg}$: 
\[ \label{eq:dKE} d(x,y) = \inf\{ \ell(\gamma)\, |\, \gamma\text{ is a smooth curve
    in } X^{reg} \text{ from } x \text{ to } y\}, \]
where $\ell(\gamma)$ denotes the length of $\gamma$ with respect to
$\omega$. We define $(\hat{X}, d_{\hat{X}})$ to be the metric completion of
$(X^{reg}, d)$, and we extend the volume form $\omega^n$ to
$\hat{X}$ trivially. In this way $(\hat{X}, d_{\hat{X}}, \omega^n)$ defines
a metric measure space. The complement of $X^{reg}$ has zero capacity,
by the following result, due to Sturm~\cite{SturmPrivate} (see also
Song~\cite[Lemma 3.7]{Song14}). 
\begin{lemma}\label{lem:cutoff}
  There is a sequence of smooth functions $\phi_i : X^{reg}\to [0,1]$
  with compact support, such that we have: for any compact $A\subset
  X^{reg}$
  we have $\phi_i|_A = 1$ for sufficiently large $i$, and
  \[ \lim_{i\to\infty} \int_{X^{reg}} |\nabla\phi_i|^2\, \omega^n
    = 0. \]
\end{lemma}

From this we have the following. 
\begin{lemma}
  $(\hat{X}, d_{\hat{X}}, \omega^n)$ defines a $2n$-dimensional almost
  smooth measure metric space in the sense of
  Definition~\ref{defn:almostsmooth}. 
\end{lemma}
\begin{proof}
  The open set $\Omega\subset \hat{X}$ is the smooth locus
  $X^{reg}$ viewed as a subset of its metric completion $\hat{X}$, equipped with
  the K\"ahler  metric $\omega$. The conditions (1) and
  (2) in Definition~\ref{defn:almostsmooth} are automatically
  satisfied.   The fact that $\hat{X}\setminus X^{reg}$ has capacity zero follows
  from the existence of good cutoff functions in Lemma~\ref{lem:cutoff}. 
\end{proof}

In order to show that $\hat{X}$ is an $RCD$ space, we will
use the characterization of $RCD$ spaces in
Honda~\cite[Corollary 3.10]{Honda} (see also
Ambrosio-Gigli-Savar\'e~\cite{AGS}). We state this Corollary here
in our setting. Note that our notion of almost smooth metric measure
space is slightly different from that in \cite{Honda}. 
\begin{cor}[See \cite{Honda}] \label{cor:Honda}
  The metric completion $(\hat{X}, d_{\hat{X}}, \omega^n)$ is an
  $RCD(K, 2n)$ space, where $K\in \mathbb{R}$, if 
   it is an almost smooth compact metric measure space
  associated with $X^{reg}$ in the sense of \cite[Definition
  3.1]{Honda}, and the following conditions hold:
  \begin{enumerate}
  \item The Sobolev to Lipschitz property holds, that is any $f\in
    W^{1,2}(\hat{X})$, with $|\nabla f|(x)\leq 1$ for
    $\omega^n$-almost every $x$, has a $1$-Lipschitz
    representative;
  \item The $L^2$-strong compactness condition holds, that is the
    inclusion $W^{1,2}(\hat{X})\hookrightarrow L^2(\hat{X})$ is a
    compact operator;
  \item Any $W^{1,2}$-eigenfunction of the Laplacian on $\hat{X}$ is Lipschitz;
  \item $\mathrm{Ric}(\omega) \geq K\omega$ on $X^{reg}$. 
  \end{enumerate}
\end{cor}
In these conditions note that for almost smooth compact metric measure
spaces the Sobolev space $W^{1,2}(\hat{X})$ defined using the
Riemannian structure on $X^{reg}$
coincides with the $H^{1,2}(\hat{X}, d_{\hat{X}}, \omega^n)$-space defined
using the Cheeger energy (see \cite[Proposition 3.3]{Honda}).
\begin{proof}
  The only place where the difference between our notion of capacity
  zero in Definition~\ref{defn:almostsmooth} and Honda's notion plays
  a role is in the proof of \cite[Theorem
3.7]{Honda} to deduce Equation (3.13), stating that the Hessian of
$f_N$ is in $L^2$ (see \cite{Honda} for the meaning of $f_N$). 
We can also deduce this by using cutoff functions
that satisfy our Condition (3b) in Definition~\ref{defn:almostsmooth}.
To
simplify the notation we will write $\Omega = X^{reg}$. 
Let us recall Equation
(3.12) from \cite{Honda}, which in our notation states
\[ \label{eq:H10} \frac{1}{2}\int_{\Omega} |\nabla f_N|^2 \Delta \phi_i^2\,
  \omega^n \geq \int_{\Omega} \phi_i^2 \Big(
  |\mathrm{Hess}_{f_N}|^2 + \langle \nabla\Delta f_N, \nabla
  f_N\rangle + K|\nabla f_N|^2\Big)\, \omega^n, \]
where $\mathrm{Ric}(\omega) \geq K \omega$, and we used
$\phi_i^2$ as the cutoff function instead of $\phi_i$. Note that $0
\leq \phi_i^2 \leq 1$, and $\nabla \phi_i^2 = 2\phi_i \nabla\phi_i$,
so $\phi_i^2$ satisfies the same estimate as $\phi_i$. In addition
$f_N$ is a Lipschitz function such that $f_N, \Delta f_N \in
W^{1,2}$. We have
\[ \int_{\Omega} |\nabla f_N|^2 \Delta \phi_i^2\, \omega^n &=
  -\int_{\Omega} 4|\nabla f_N| \phi_i \langle\nabla |\nabla f_N|,
  \nabla\phi_i\rangle\, \omega^n \\
  &\leq \int_{\Omega} \left( \phi_i^2
    |\mathrm{Hess}_{f_N}|^2 + 4|\nabla f_N|^2 |\nabla\phi_i|^2\right)\, \omega^n.
\]
It follows using this in \eqref{eq:H10} that
\[ \int_{\Omega}  \frac{1}{2}\phi_i^2 |\mathrm{Hess}_{f_N}|^2 \,
  \omega^n &\leq \int_{\Omega} \Big( 2|\nabla f_N|^2
    |\nabla\phi_i|^2 - \phi_i^2 \langle \nabla\Delta f_N, \nabla
    f_N\rangle \\ &\qquad \qquad - \phi_i^2 K |\nabla f_N|^2\Big) \,
    \omega^n.  \]
Letting $i\to \infty$ and using that $|\nabla f_N| \in L^\infty$, we
obtain that
\[ \int_{\Omega} |\mathrm{Hess}_{f_N}|^2\, \omega^n < \infty. \]
The rest of the argument is the same as in \cite[Theorem
3.7]{Honda}. 
\end{proof}

Note that in our setting we have the following. In Section~\ref{sec:singularKERCD} 
we will show the remaining Condition (3) in the setting of
Theorem~\ref{thm:main}.

\begin{prop}\label{prop:RCD2}
  The metric measure space $(\hat{X}, d_{\hat{X}}, \omega^n)$ satisfies 
  Conditions (1), (2) and (4) in Corollary~\ref{cor:Honda}, for some
  $K\in \mathbb{R}$. 
\end{prop}
\begin{proof}
  Condition (4) is satisfied by definition. To verify Condition
  (1), let $f\in W^{1,2}(\hat{X})$, such that $|\nabla f|(x) \leq 1$ for
  $\omega^n$-almost every $x$. On $X^{reg}$ the Sobolev to
  Lipschitz property holds, so we can assume that $f$ is 1-Lipschitz
  on $X^{reg}$. By the definition of the distance $d$, this implies
  that for any $x,y\in X^{reg}$ we have $|f(x) - f(y)| \leq |x-y|$. We
  can then extend $f$ uniquely to the completion $\hat{X}$ so that the same
  condition continues to hold.
  Condition (2) follows from the Sobolev
  inequality shown by Guo-Phong-Song-Sturm~\cite[Theorem
  2.1]{GPSS2}.
\end{proof}

Let us recall from De Philippis-Gigli~\cite{DPG} that an
$RCD(K,N)$-space $(Z, d, \mathfrak{m})$ is called non-collapsed, if the $N$-dimensional
Hausdorff measure on $(Z, d)$ agrees with $\mathfrak{m}$. In
particular, if an $n$-dimensional almost smooth metric measure space in
Definition~\ref{defn:almostsmooth} satisfies the $RCD(K, n)$-property,
then it is automatically non-collapsed.
Non-collapsed RCD spaces satisfy many of the properties enjoyed by
non-collapsed Ricci limits spaces studied by
Cheeger-Colding~\cite{CC1}. We will now recall some results that we will
use.

De Philippis-Gigli~\cite{DPG2} showed that in a non-collapsed $RCD(K,N)$-space
$(Z, d, \mathfrak{m})$, the tangent cones at every point $z\in Z$ are
metric cones. In \cite{DPG} they then showed that $Z$ admits a stratification
\[ S_0\subset S_1 \subset \ldots \subset S_{N-1} \subset Z, \]
where $S_k$ denotes the set of points $z\in Z$ where no tangent cone
splits off an isometric factor of $\mathbb{R}^{k+1}$, and the strata
satisfy the Hausdorff dimension estimate $\dim_{\mathcal{H}} S_k \leq
k$. Note that in contrast with the setting of non-collapsed Ricci
limit spaces, it is not necessarily the case that $S_{N-1} = S_{N-2}$,
since a non-collapsed $RCD$-space can have boundary. In our setting
however we have the following, which is a consequence of
Bru\`e-Naber-Semola~\cite[Theorem 1.2]{BNS}.

\begin{prop}\label{prop:nocodim1}
 Suppose that $(Z, d, \mathfrak{m})$ is a non-collapsed $RCD(K,
 N)$-space, and also an $N$-dimensional almost smooth metric measure
 space. Then $S_{N-1}=S_{N-2}$. Moreover any iterated tangent cone
 $Z'$ of $Z$ also satisfies $S_{N-1}= S_{N-2}$. 
\end{prop}
\begin{proof}
  Using the notation of \cite{BNS} we define $\partial Z =
  \overline{S_{N-1}\setminus S_{N-2}}$ to be the boundary of
  $Z$. Let $\Omega\subset Z$ denote the smooth Riemannian manifold in
  the definition of almost smooth metric measure space. For $z\in
  \Omega$ the tangent cones are all $\mathbb{R}^N$, so $\partial
  Z\subset Z\setminus \Omega$. In particular $\partial Z$ has capacity
  zero. Using \cite[Theorem 1.2(i)]{BNS} this implies that we must
  have $\partial Z = \emptyset$. If an iterated tangent cone $Z'$
  satisfied $\partial Z' \not= \emptyset$, then by \cite[Theorem
  1.2(i)]{BNS} we would have $\partial Z\not=\emptyset$, which
  is a contradiction as above. 
\end{proof}

We will be working with harmonic functions on RCD spaces, so we review
some basic results. Let us suppose that $(Z, d, \mathfrak{m})$ is a
non-collapsed $RCD(K,N)$-space that is also an $N$-dimensional almost
smooth metric measure space. A function $f : U \to \mathbb{R}$ on an
open set $U\subset Z$ is defined to be harmonic if $f\in
W^{1,2}_{loc}(U)$, and for any Lipschitz function
$\psi:U\to\mathbb{R}$ with compact support we have
\[ \int_U \nabla f\cdot \nabla \psi\, d\mathfrak{m} = 0. \]
Note that in our setting the integration can be taken over $U\cap
\Omega$, where $\Omega\subset Z$ is the dense open set in
Definition~\ref{defn:almostsmooth} since $Z\setminus \Omega$ has
measure zero. We will use the following result several times.

\begin{lemma} \label{lem:bddharmonic}
  Let $u: U \to \mathbb{R}$ for an open set $U\subset Z$, such that
  $u\in L^\infty(U)$. Suppose that $\Delta u = 0$ on $U\cap \Omega$,
  using the smooth Riemannian structure on $\Omega$. Then $u$ is
  harmonic on $U$.
\end{lemma}
\begin{proof}
  Let $\phi_i$ be functions as in Condition (3) of
  Definition~\ref{defn:almostsmooth}, and $\psi$ a Lipschitz function
  with compact support in $U$. We have
  \[ \int_U \psi^2 \phi_i^2 |\nabla u|^2\, d\mathfrak{m} &= -2\int_U \psi^2
    \phi_i u 
    \nabla\phi_i\cdot \nabla u\, d\mathfrak{m} - 2\int_U \phi_i^2 \psi u \nabla u\cdot \nabla\psi\,
    d\mathfrak{m} \\
    &\leq \frac{1}{2}\int_U \psi^2 \phi_i^2 |\nabla u|^2\, d\mathfrak{m} +
    4\int_U \psi^2 u^2 |\nabla\phi_i|^2\, d\mathfrak{m} + C_\psi \int_U
    u^2\, d\mathfrak{m},
  \]
  where $C_\psi$ depends on $\sup_{U\cap \Omega} |\nabla\psi|$. 
  Letting $i\to\infty$, we obtain that $u\in W^{1,2}_{loc}(U)$.

  At the same time we have
  \[ \int_U \phi_i^2 \nabla u\cdot \nabla\psi\, d\mathfrak{m} &= -2\int_U
    \phi_i \psi \nabla\phi_i\cdot \nabla u \, d\mathfrak{m} \\
  &\leq \int_U |\nabla\phi_i|^2\, d\mathfrak{m} +
  \int_{\mathrm{supp}(\nabla\phi_i)} \psi^2 |\nabla
  u|^2\,d\mathfrak{m}.\]
  Letting $i\to\infty$ we get $\int_U\nabla u\cdot\nabla\psi = 0$, so
  $u$ is harmonic on $U$. 
\end{proof}

We will also need the following gradient estimate, generalizing
Cheng-Yau's gradient estimate.
\begin{prop}[Jiang~\cite{Jiang14}, Theorem 1.1]
  \label{prop:RCDgrad}
  Let $u$ be a harmonic function on a ball $B(p, 2R)$ in an $RCD(N,
  K)$-space. There is a constant $C=C(R,N,K)$ such that
  \[ \sup_{B(p,R)} |\nabla u| \leq C \fint_{B(p, 2R)}
    |u|\, d\mathfrak{m}. \]  
\end{prop}
Note that a similar estimate holds for solutions of $\Delta u = c$ on
$U\subset Z$ for a
constant $c$, by considering $u - ct^2/2$ on the space $U\times
\mathbb{R}_t$.

\section{The RCD property of singular K\"ahler-Einstein spaces}
\label{sec:singularKERCD}
The main result in this section will be that the completion of the
K\"ahler-Einstein metric on $X^{reg}$ in Theorem~\ref{thm:main}
defines a non-collapsed RCD space. We will first need some estimates
for the cscK approximations of $(X,\omega_{KE})$. 

\subsection{Constant scalar curvature approximations}
Let $(X, \omega_{KE})$ be a singular K\"ahler-Einstein space,
where $\mathrm{Ric}_{\omega_{KE}} = \lambda\omega_{KE}$. Suppose that
$(X, \omega_{KE})$
can be approximated by cscK metrics as in
Definition~\ref{defn:cscKapprox}. In particular there is a resolution
$Y$ of $X$, that admits a family of cscK metrics $\omega_\epsilon$ in
suitable K\"ahler classes $[\eta_\epsilon]$, such that the
$\eta_\epsilon$ converge to $\pi^*\eta_X$. Here $\eta_X$ is a smooth
metric on $X$ in the sense that it is the restriction of a smooth
metric under local embeddings into $\mathbb{C}^N$. 

We will need the following, which is immediate from the work of
Guo-Phong-Song-Sturm~\cite[Theorem 2.2]{GPSS2}.
\begin{thm}\label{thm:Hupper}
  Let $H(x,y,t)$ denote the heat kernel on $(Y, \omega_\epsilon)$. There is a
  continuous function $\bar{H}:(0,2] \to \mathbb{R}$, depending on
  $(X,\omega_{KE})$, but independent of $\epsilon$, such that we have the
  upper bound
  \[ H(x,y,t) \leq \bar{H}(t), \quad \text{ for } x,y\in Y \text{ and } t\in
    (0,2]. \]
  Note that $\bar{H}(t)\to\infty$ as $t\to 0$. 
\end{thm}

In addition the constant scalar curvature metrics $\omega_\epsilon$ satisfy
the following integral bounds for their Ricci curvatures. We will use
these integral bounds as a replacement for having lower bounds for
the Ricci curvature, when we approximate $\omega_{KE}$ with
$\omega_\epsilon$. 
\begin{prop}\label{prop:cscKbounds}
  Let us define $\tRic_\omega = \mathrm{Ric}_\omega - \lambda\omega$.
  We have the following estimates: 
  \[ \lim_{\epsilon\to 0} \int_{Y} |\tRic_{\omega_\epsilon}|^2 +
    \frac{|\nabla\tRic_{\omega_\epsilon}|^2}{(1 + |\tRic_{\omega_\epsilon}|^2)^{1/2}} +
    |\Delta(1+|\tRic_{\omega_\epsilon}|^2)^{1/2}|\, \,\omega_\epsilon^n = 0, \]
  and
  \[\int_Y |\mathrm{Rm}_{\omega_\epsilon}|^2\, \omega_\epsilon^n <
  C, \]
  for $C$ independent of $\epsilon$. 
\end{prop}
\begin{proof}
  First recall the well-known result of Calabi~\cite{Calabi}
  relating the $L^2$-norms of the scalar curvature, the Ricci and Riemannian
  curvature tensors of a K\"ahler
  metric. Let us denote by $R, \mathrm{Ric}, \mathrm{Rm}$ the scalar
  curvature, the Ricci form and the Riemannian curvature tensor. Since
  $R_{\omega_\epsilon}$ is constant, we have
  \[ R_{\omega_\epsilon} = \frac{ 2n \pi  c_1(Y) \cup
      [\omega_\epsilon]^{n-1}}{[\omega_\epsilon]^n}. \]
  Note that we have
  \[ \label{eq:limR} \lim_{\epsilon\to 0} \frac{ 2n \pi  c_1(Y) \cup
      [\omega_\epsilon]^{n-1}}{[\omega_\epsilon]^n} = \frac{2n\pi c_1(X) \cup
      [\omega_{KE}]^{n-1}}{[\omega_{KE}]^n} = n\lambda, \]
  since $[\omega_{KE}]^{n-1}$ vanishes when paired with the
  exceptional divisor of the map $Y\to X$. 
  In addition
  \[    \int_{Y} |\mathrm{Ric}_{\omega_\epsilon}|^2\,
    \omega_\epsilon^n &= R_{\omega_\epsilon}^2 
    [\omega_\epsilon]^n - 4n(n-1)\pi^2 c_1(Y)^2\cup [\omega_\epsilon]^{n-2}, \\
     \int_Y (|\mathrm{Ric}_{\omega_\epsilon}|^2 -
     |\mathrm{Rm}_{\omega_\epsilon}|^2)\, \omega_\epsilon^n &= n(n-1)\big( 4\pi^2
     c_1(Y)^2 - 8\pi^2 c_2(Y)\big) \cup [\omega_\epsilon]^{n-2}. \]
   Since the cohomology classes $[\omega_\epsilon]=[\eta_\epsilon]$ are uniformly
   bounded, and in addition $[\omega_\epsilon]^n \geq [\eta_X]^n > 0$, it
   follows that $R_{\omega_\epsilon}$, and the $L^2$ norms of
   $|\mathrm{Ric}_{\omega_\epsilon}|, |\mathrm{Rm}_{\omega_\epsilon}|$ are all
   uniformly bounded, independently of $\epsilon$.

   To see the first claim in the Proposition, note that
   \[ \int_Y |\mathrm{Ric}_{\omega_\epsilon} - \lambda \omega_\epsilon|^2\,
     \omega_\epsilon^n = (R_{\omega_\epsilon} - n\lambda)^2 [\omega_\epsilon]^n - n(n-1)
     \big( 2\pi c_1(Y) - \lambda[\omega_\epsilon]\big)^2\cup [\omega_\epsilon]^{n-2}. \]
   As $\epsilon\to 0$, this converges to zero by \eqref{eq:limR} and the
   fact that $2\pi c_1(X) = \lambda [\omega_{KE}]$.

   To estimate $\nabla\tRic_{\omega_\epsilon}$ and $\Delta \tRic_{\omega_\epsilon}$ note
   that we have the following equation satisfied by any constant
   scalar curvature metric:
   \[ \Delta |\tRic|^2 &= \nabla_k\nabla_{\bar k}
     (\tRic_{p\bar q} \tRic_{q\bar p}) \\
     &= 2|\nabla_k \tRic_{p\bar q}|^2 + \mathrm{Rm}\ast
     \tRic\ast \tRic, \]
   where $\ast$ denotes a tensorial contraction.  It follows that 
   \[ \label{eq:DtRic}
      \Delta (1 + |\tRic|^2)^{1/2} &= (1+|\tRic|^2)^{-1/2}\left(
       |\nabla\tRic|^2 - |\nabla|\tRic||^2 +
       \mathrm{Rm}\ast\tRic\ast\tRic\right) \\
     &\qquad + \frac{|\nabla |\tRic||^2}{(1 +
         |\tRic|^2)^{3/2}}. \]
     For a constant scalar curvature K\"ahler metric the form $\tRic$ is
     harmonic, so we have the following refined Kato inequality (see Branson~\cite{Bran},
     Calderbank-Gauduchon-Herzlich~\cite{CGH} or
     Cibotaru-Zhu~\cite[Theorem 3.8]{CZ}):
     \[ |\nabla|\tRic||^2 \leq \alpha_n |\nabla \tRic|^2 \]
     for a dimensional constant $\alpha_n < 1$. It follows from
     \eqref{eq:DtRic} that
     \[ \Delta (1 + |\tRic|^2)^{1/2} \geq (1-\alpha_n) \frac{
         |\nabla\tRic|^2}{(1 + |\tRic|^2)^{1/2}} - C |\mathrm{Rm}|\,
       |\tRic|. \]
     Integrating over $Y$, we get
     \[ \int_Y \frac{|\nabla\tRic_{\omega_\epsilon}|^2}{(1 +
         |\tRic|_{\omega_\epsilon}^2)^{1/2}}\, \omega_\epsilon^n \leq
       C_1 \Vert \mathrm{Rm}_{\omega_\epsilon}\Vert_{L^2} \Vert
       \tRic_{\omega_\epsilon} \Vert_{L^2} \to 0, \]
       as $\epsilon\to 0$. It then follows from \eqref{eq:DtRic} that
       \[ \int_Y |\Delta (1 + |\tRic_{\omega_\epsilon}|^2)^{1/2}| \,
         \omega_\epsilon^n \leq \int_Y \frac{|\nabla\tRic_{\omega_\epsilon}|^2}{(1 +
         |\tRic|_{\omega_\epsilon}^2)^{1/2}}\, \omega_\epsilon^n  + C \Vert \mathrm{Rm}_{\omega_\epsilon}\Vert_{L^2} \Vert
       \tRic_{\omega_\epsilon} \Vert_{L^2} \to 0,\]
     as $\epsilon\to 0$. 
\end{proof}

\subsection{Proof of the RCD property} 
In this section we assume that $(X,\omega_{KE})$ is a
singular K\"ahler-Einstein space, with $\mathrm{Ric}_{\omega_{KE}} =
\lambda \omega_{KE}$,  that can be approximated with cscK
metrics as in Definition~\ref{defn:cscKapprox}. Our first result is the following.
\begin{prop} \label{prop:RCD}
  The metric completion $(\hat{X}, d, \omega_{KE}^n)$ is an
  $RCD(\lambda, 2n)$ space.
\end{prop}
\begin{proof}
From Proposition~\ref{prop:RCD2} it follows that it is
sufficient to check condition (3) in
Corollary~\ref{cor:Honda}, i.e. to 
show that the eigenfunctions of the Laplacian on
$\hat{X}$ are bounded. More precisely, suppose that
$u\in W^{1,2}(\hat{X})$ satisfies $\Delta u = -b 
u$ on $X^{reg}$ for a constant $b$. We will show that
then $|\nabla u| \in L^\infty(X^{reg})$. 

For simplicity we can assume that $\Vert u\Vert_{L^2} = 1$. Using that
$u\in W^{1,2}(\hat{X})$, and also \cite[Lemma 11.2]{GPSS2}, we have
  \[ \sup |u| + \int_{X^{reg}} |\nabla u|^2 \omega_{KE}^n < C, \]
  where $C$ could depend on $u$ (in particular on $b$). 
  
  Next we will use the approximating cscK metrics $\omega_\epsilon$
  on the resolution $Y$ of $X$. Let us fix a large $i$, and let $f = \phi_iu$ for the
  cutoff function $\phi_i$ in Lemma~\ref{lem:cutoff}. We can view $f$
  as a function on $Y$, supported away from the exceptional divisor,
  where the metrics $\omega_\epsilon$ converge smoothly to
  $\omega_{KE}$. Note that we have a uniform bound $\sup|f| < C$,
  and also
  \[ \int_Y |\nabla f|^2 \omega_\epsilon^n \leq \int_Y 2( |u \nabla\phi_i|^2 +
    |\phi_i \nabla u|^2)\, \omega_\epsilon^n < 2C, \]
  for sufficiently small $\epsilon$. 

  Let us fix a point $x_0\in X$ where $\phi_i(x_0)=1$. We can view
  $x_0\in Y$ too. We will do the following calculation on $Y$, using
  the metric $\omega_\epsilon$ for sufficiently small $\epsilon$. To
  simplify the notation we will omit the subscript $\epsilon$. All
  geometric quantities are defined using the metric
  $\omega_\epsilon$. We will write $\rho_t = H(x_0,y,t)$ for the heat
  kernel centered at $x_0$ on $(Y, \omega_\epsilon)$, and let $f_t$
  denote the solution of the heat equation on $(Y, \omega_\epsilon)$
  with initial condition $f$. We will also omit the volume form
  $\omega_\epsilon^n$ in the integrals below. We have the following.
  \[ \label{eq:dsdf} \partial_s \int_Y \frac{1}{2}|\nabla f_{t-s}|^2\, \rho_s &=
    \int_Y -\langle \nabla f_{t-s}, \nabla \Delta f_{t-s}\rangle
    \rho_s + \frac{1}{2}|\nabla f_{t-s}|^2\, \Delta \rho_s \\
    &= \int_Y \Big( |\nabla^2 f_{t-s}|^2 + \mathrm{Ric}(\nabla
    f_{t-s}, \nabla f_{t-s}) \Big) \rho_s. 
  \]
  In order to compensate for the Ricci term, we let $\psi^2 = (1 +
  |\tRic|^2)^{1/2}$, where $\tRic = \mathrm{Ric}_{\omega_\epsilon} - \lambda
  \omega_\epsilon$ as in Proposition~\ref{prop:cscKbounds}. We have
  \[ \label{eq:dspsif} \partial_s \int_Y \psi^2 f_{t-s}^2\, \rho_s &= \int_Y -2\psi^2
    f_{t-s} \Delta f_{t-s} \,\rho_s + \psi^2 f_{t-s}^2 \Delta \rho_s \\
    &= \int_Y \Big(\Delta(\psi^2) f_{t-s}^2 + 2\langle \nabla\psi^2,
    \nabla f^2_{t-s}\rangle + 2\psi^2 |\nabla f_{t-s}|^2\Big)\,
    \rho_s \\
    &\geq -C\int_Y (|\Delta \psi^2|  + |\nabla\psi|^2)\, \rho_s + \int_Y
    \psi^2 |\nabla f_{t-s}|^2\, \rho_s,
  \]
  where the constant $C$ depends on the uniform supremum bound for $f_{t-s}$. 

  Note that $\psi^2 \geq |\mathrm{Ric}| - n|\lambda|$, so if we combine
  \eqref{eq:dsdf} and \eqref{eq:dspsif}, we get
  \[ \partial_s \int_Y \left( \frac{1}{2} |\nabla f_{t-s}|^2 + \psi^2
      f_{t-s}^2\right)\, \rho_s \geq -C\int_Y (|\Delta\psi^2| +
    |\nabla\psi|^2)\, \rho_s - \int_Y n|\lambda| |\nabla f_{t-s}|^2\,
    \rho_s. \]
  At this point, let us fix $s_0 > 0$, and only work with $s \in [s_0,
  2]$. From Proposition~\ref{prop:cscKbounds} we know that
  $\Vert\Delta\psi^2 \Vert_{L^1}, \Vert \nabla\psi\Vert_{L^2} \to 0$
  as $\epsilon\to 0$. From Theorem~\ref{thm:Hupper} we have a uniform
  upper bound for $\rho_s$, depending on $s_0$, but independent of
  $\epsilon$. Therefore, if we choose $\epsilon$ sufficiently small,
  say $\epsilon < \epsilon_{s_0}$, then we have
  \[ \partial_s \int_Y \left( \frac{1}{2} |\nabla f_{t-s}|^2 + \psi^2
      f_{t-s}^2\right)\, \rho_s \geq -1 - n|\lambda| \int_Y |\nabla
    f_{t-s}|^2\, \rho_s, \]
  and so
  \[ \partial_s\, e^{2n|\lambda|s} \int_Y \left( \frac{1}{2} |\nabla f_{t-s}|^2 + \psi^2
      f_{t-s}^2\right)\, \rho_s \geq -C. \]
  Applying this with $t=1+s_0$ and integrating from $s=s_0$ to
  $s=1+s_0$, it follows that for such $\epsilon$ we have 
  \[ e^{2n|\lambda| s_0} \int_Y \left(\frac{1}{2} |\nabla f_1|^2 + \psi^2
      f_1^2\right)\, \rho_{s_0} \leq C + e^{2n|\lambda|(s_0+1)} \int_Y
    \left(\frac{1}{2}|\nabla f|^2 + \psi^2 f^2\right)\,
    \rho_{1+s_0}. \]
  Using the uniform upper bound for $\rho_{1+s_0}$, together with the
  integral bound for $|\tRic|^2$ from
  Proposition~\ref{prop:cscKbounds}, we obtain that
  \[ \int_Y |\nabla f_1|^2\, \rho_{s_0} \leq C, \]
  where $C$ is independent of $\epsilon, s_0$. As $\epsilon\to 0$, the heat
  kernels $\rho_{s_0}$ converge locally smoothly on $X^{reg}$ to the
  heat kernel on $(\hat{X}, \omega_{KE})$, and so in the limit we
  obtain the estimate
  \[ \int_{X^{reg}} |\nabla f_1|^2\, \rho_{s_0} \leq C, \]
  where all the quantities are computed using $\omega_{KE}$, and
  recall that $f_1$ is simply the solution $f_t$ of the heat flow with initial condition $f$
  at time $t=1$. Note that the constant $C$ does not depend on $s_0$,
  so in fact, by letting $s_0\to 0$, we obtain the pointwise estimate
  \[ |\nabla f_1|^2(x_0) \leq C, \]
  and this holds uniformly for any $x_0\in X^{reg}$.

  Recall that $f= \phi_i u$, where $u$ is the eigenfunction that we
  want to estimate, and $\phi_i$ is a cutoff function from
  Lemma~\ref{lem:cutoff}. To keep track of the dependence on $i$, let
  us now write $f^{(i)}= \phi_i u$, and write $f^{(i)}_1$ for the
  corresponding solutions of the heat equation at time $1$. Since
  $f^{(i)} \to u$ in $L^2$, it follows that for any compact set
  $K\subset X^{reg}$ the solutions $f^{(i)}_1$ converge smoothly to
  $u_1$ on $K$. But $u_1 = e^{-b}u$, so we obtain the required pointwise
  bound $|\nabla u|^2(x_0) \leq e^{2b} C$ for any $x_0\in X^{reg}$. 
\end{proof}

Next we show that singular K\"ahler-Einstein metrics on projective varieties, that can be
approximated by cscK metrics, define  K\"ahler
currents. This result was previously shown by
Guedj-Guenancia-Zeriahi~\cite{GGZ} for singular
K\"ahler-Einstein metrics that are either globally smoothable, or that
only have isolated smoothable singularities. 
\begin{thm}\label{thm:Kcurrent}
  Let $\omega_{KE}$ denote a singular K\"ahler-Einstein metric on a
  normal projective variety $X$, which can be approximated by cscK
  metrics as in Definition~\ref{defn:cscKapprox}. Let $\eta_{FS}$ denote the
  pullback of the Fubini-Study metric to $X$ under a projective
  embedding of $X$. Then there is a constant $\delta
  > 0$ such that $\omega_{KE} > \delta \eta_{FS}$.  
\end{thm}
\begin{proof}
  By assumption we have cscK metrics $\omega_\epsilon =
  \eta_\epsilon + \ddbar u_\epsilon$ on a resolution $\pi:Y \to X$,
  where $\eta_\epsilon\to \pi^*\eta_X$ for a smooth metric $\eta_X$ on
  $X$, where $\eta_\epsilon \geq \pi^*\eta_X$. 
  We apply the Chern-Lu inequality to the map $\pi : Y \to X$, away
  from the exceptional divisor $E$, where on $Y$ we use the metric
  $\omega_\epsilon$ and on $X$ we use the pullback $\eta_{FS}$ of the
  Fubini-Study metric under a projective embedding of $X$. For simplicity we write $\eta_{FS}$
  for $\pi^*\eta_{FS}$, and we write $g_{i\bar j}$ and $h_{i\bar j}$ for
  the metric components of $\omega_\epsilon$ and $\eta_{FS}$
  respectively. On $Y\setminus E$ we then have $|\partial \pi|^2 =
  \mathrm{tr}_{\omega_\epsilon} \eta_{FS}$, and (see e.g. \cite{JMR})
  \[ \Delta_{\omega_\epsilon} \log \mathrm{tr}_{\omega_\epsilon}
    \eta_{FS} \geq \frac{ g^{i\bar l}g^{k\bar j}
      \mathrm{Ric}(\omega_\epsilon)_{i\bar j} h_{k\bar
        l}}{\mathrm{tr}_{\omega_\epsilon} \eta_{FS}} - A
    \mathrm{tr}_{\omega_\epsilon} \eta_{FS}, \]
  where $A$ is independent of $\epsilon$, using that $\eta_{FS}$ has
  bisectional curvature bounded from above. It follows that
  \[ \Delta_{\omega_\epsilon} \log \mathrm{tr}_{\omega_\epsilon}
    \eta_{FS} &\geq \frac{ g^{i\bar l}g^{k\bar j}
      (\mathrm{Ric}(\omega_\epsilon)_{i\bar j} - \lambda g_{i\bar j})h_{k\bar
        l}}{\mathrm{tr}_{\omega_\epsilon} \eta_{FS}} + \lambda - A
    \mathrm{tr}_{\omega_\epsilon} \eta_{FS} \\
     &\geq -|\mathrm{Ric}_{\omega_\epsilon} - \lambda\omega_\epsilon|
     +\lambda - A\mathrm{tr}_{\omega_\epsilon} \eta_X.
   \]
   We also have
   \[ \Delta_{\omega_\epsilon} (-u_\epsilon) =
     \mathrm{tr}_{\omega_\epsilon} \eta_\epsilon - n \geq
     \mathrm{tr}_{\omega_\epsilon} \eta_X - n \geq C_1^{-1}
     \mathrm{tr}_{\omega_\epsilon} \eta_{FS} - n, \]
   for some $C_1 > 0$, using that locally both $\eta_{FS}$ and $\eta_X$
   are given by pullbacks of smooth metrics under embeddings of $X$. 
   This implies, that
   \[ \Delta_{\omega_\epsilon} ( \log \mathrm{tr}_{\omega_\epsilon}
    \eta_{FS}  - AC_1u_\epsilon) &\geq -|\mathrm{Ric}_{\omega_\epsilon} - \lambda\omega_\epsilon|
    +\lambda - AC_1n \\
    &\geq -|\mathrm{Ric}_{\omega_\epsilon} - \lambda\omega_\epsilon| -
    C_2,  \]
  for some $C_2 > 0$. 
  Let us define
  \[ F = \max\{0, \log \mathrm{tr}_{\omega_\epsilon}
    \eta_{FS}  - AC_1u_\epsilon\}. \]
  Since $\omega_\epsilon$ is a K\"ahler metric, $F$ is bounded from
  above, and by definition $F$ is also bounded below. In addition $F$
  satisfies the differential inequality
  \[ \Delta_{\omega_\epsilon} F \geq - |\mathrm{Ric}_{\omega_\epsilon}
    - \lambda\omega_\epsilon| - C_2 \]
  in a distributional sense on all of $Y$. To see this, note first
  that the differential inequality is satisfied in the distributional
  sense on $Y\setminus E$ by the definition of $F$ as a maximum of two
  functions satisfying the inequality. Then the differential
  inequality can be extended across $E$ using that $F$ is bounded, by
  an argument similar to Lemma~\ref{lem:bddharmonic}. 
  
    Fix $x\in Y\setminus E$, and let $H(x,y,t)$ denote the heat kernel
    on $(Y, \omega_\epsilon)$. Fix some $t_0 > 0$. For $t\in [t_0, 1]$
    we have
    \[ \partial_t \int_Y F(y)\, H(x,y,t)\, dy &= \int_Y F(y)\, \Delta_y H(x,y,t)\, dy \\
    &= \int_Y \Delta_y F(y) \, H(x,y,t)\, dy \\
    &\geq \int_Y (-|\mathrm{Ric}_{\omega_\epsilon} -
    \lambda\omega_{\epsilon}|(y) -C_2) H(x,y,t)\, dy. \]
    Using the uniform upper bound for $H$ (see Theorem~\ref{thm:Hupper}), together with
    Proposition~\ref{prop:cscKbounds}, we find that there exists an
    $\epsilon_0 = \epsilon_0(t_0)$, depending on $t_0$, such that if
    $\epsilon < \epsilon_0$, then
    \[ \partial_t \int_Y F(y)\, H(x,y,t)\, dy \geq -2C_2, \]
    and so for $\epsilon < \epsilon_0$ we have
    \[ \int_Y F(y) H(x,y,t_0)\, dy &\leq \int_Y F(y) H(x,y,1)\, dy +
      2C_2. \]
    Note that
    \[ F \leq e^{-AC_1u_\epsilon} \mathrm{tr}_{\omega_\epsilon}
      \eta_{FS}, \]
    so we have (using the uniform upper bound for the heat kernel as
    well), 
    \[ \int_Y F(y) H(x,y,t_0)\, dy &\leq C_3 e^{AC_1\sup |u_\epsilon|} \int_Y
      \mathrm{tr}_{\omega_\epsilon}\eta_{FS}\, \,\omega_\epsilon^n + 2C_2  \\
      &\leq C_4. \]
    Here we also used that we have a uniform bound for $\sup|u_\epsilon|$,
    and the cohomology classes 
    $[\omega_\epsilon]$ are uniformly bounded. Crucially, the constant $C_4$ is
    independent of $t_0$. 

    Note that as $\epsilon\to 0$, the heat kernels $H(x,y,t)$ for
    $(Y,\omega_\epsilon)$ converge locally smoothly on $Y\setminus E$
    to the heat kernel for $(X, \omega_{KE})$. At the same time, the
    function $F(y)$ converges locally uniformly on $Y\setminus E$ to
    \[\max\{0,\log \mathrm{tr}_{\omega_{KE}}\eta_{FS} - AC_1u_{KE}\}.\]
    It follows that in the limit, for any $t > 0$, we have
    \[ \int_{X^{reg}} (\log\mathrm{tr}_{\omega_{KE}}\eta_{FS}- AC_1u_{KE})(y)\,
      H_{\omega_{KE}}(x,y,t) \, \omega_{KE}^n(y) \leq C_4. \]
    Letting $t\to 0$ we obtain a pointwise bound
    $\mathrm{tr}_{\omega_{KE}}\eta_{FS} < C_5$, as required. 
\end{proof}

\section{Homeomorphism with the underlying variety}
\label{sec:DS}
In this section our goal is to show that the metric completion $\hat{X}$ of the
smooth locus of a singular K\"ahler-Einstein metric $(X, \omega_{KE})$
is homeomorphic to $X$, under suitable assumptions. These assumptions
hold in the setting of Theorem~\ref{thm:main}, where $(X,\omega_{KE})$
can be approximated with cscK metrics.

We assume that $X$ is a normal projective
variety of dimension $n$, and we have a K\"ahler current $\omega$ on $X$ with bounded
local potentials, such that $\omega\in c_1(L)$ for a line bundle $L$
on $X$. We will write $\eta_{FS}$ for the pullback of the
Fubini-Study metric to $X$ under a projective embedding.
We make the following assumptions:
\begin{itemize}
\item[(1)] The Ricci form of $\omega$, as a current, satisfies
  $\mathrm{Ric}(\omega) = \lambda\omega$ for a constant $\lambda\in
  \mathbb{R}$ on the regular part $X^{reg}$ of $X$. 
 \item[(2)] $\omega$ is a K\"ahler current, i.e. $\omega \geq c
   \eta_{FS}$ on $X$ for some $c > 0$. 
  \item[(3)] The metric completion $(\hat{X}, d_{\hat{X}})$ of $(X^{reg},
    \omega)$ is a noncollapsed $RCD(2n, \lambda)$ space, where the
    measure on $\hat{X}$ is the pushforward of $\omega^n$ from
    $X^{reg}$. 
  \item[(4)] We have $\omega^n = F \eta_{FS}^n$, where $F \in L^p(X,
    \eta_{FS}^n)$ for some $p > 1$. 
\end{itemize}

We have seen that Conditions (1)--(3) are satisfied for singular
K\"ahler-Einstein metrics $(X,\omega_{KE})$, with $\omega_{KE}\in
c_1(L)$, that can be approximated with cscK metrics in the sense of 
Definition~\ref{defn:cscKapprox}. For Condition (4), see
Eyssidieux-Guedj-Zeriahi~\cite[Section 7]{EGZ}. 

The main result of this section is the following, and the proof will
be completed after Proposition~\ref{prop:nocodim2} below. 
\begin{thm}\label{thm:homeomorphic}
  Let $(X,\omega)$ satisfy the conditions (1)--(4) above. Then the
  metric completion $\hat{X}$ is homeomorphic to $X$. 
\end{thm}

Rescaling the metric $\omega$ we can assume that $L$ is a very
ample line bundle on $X$. The sections of $L$ define a holomorphic
embedding $\Phi_X : X\to \mathbb{CP}^N$, and we can identify the
image of this embedding with $X$. By the assumption that $\omega$ is a
K\"ahler current, we
have that the map
\[ \Phi_X : (X^{reg}, \omega) \to (X, \eta_{FS}) \subset
  \mathbb{CP}^N \]
is Lipschitz continuous, where we use the length metric as defined in
\eqref{eq:dKE}. In particular $\Phi_X$ extends to a Lipschitz
continuous map
\[ \hat{\Phi}_X : \hat{X} \to (X, \eta_{FS}). \]
Note that $\hat{\Phi}_X$ is surjective, since the image of $X^{reg}$
is dense in $X$, so our task is to prove that $\hat{\Phi}_X$ is
injective, i.e. to show that the sections of $L$ separate points of $\hat{X}$. In
fact we will work with $L^k$ for large $k$, however since $L$ is very
ample, the map defined by section of $L^k$ is obtained by composing
the map defined by sections of $L$ with an embedding of
$\mathbb{CP}^N$ into a larger projective space.

The general strategy for showing that sections of $L^k$ separate points of
$\hat{X}$ is similar to the work of Donaldson-Sun~\cite{DS1}. We will
apply the following form of H\"ormander's estimate (see
e.g. \cite[Theorem 6.1]{Demailly}):
\begin{thm}\label{thm:Hormander}
  Let $(P,h_P)$ be a Hermitian holomorphic
  line bundle on a K\"ahler manifold $(M, \omega_M)$, which
  admits some complete K\"ahler metric. Suppose that the curvature
  form of $h_P$ satisfies $\sqrt{-1}F_{h_P} \geq c \omega_M$ for some constant
  $c > 0$. Let $\alpha\in \Omega^{n,1}(P)$ be such that
  $\bar\partial\alpha = 0$. Then there exists $u\in \Omega^{n,0}(P)$
  such that $\bar\partial u = \alpha$, and
  \[ \Vert u\Vert_{L^2}^2 \leq \frac{1}{c} \Vert
    \alpha\Vert^2_{L^2}, \]
  provided the right hand side is finite. 
\end{thm}

We will apply this result to $M=X^{reg}$, with the metric
$\omega_M=k\omega$. Note that it follows from Demailly~\cite[Theorem
0.2]{Dem82},  that $X^{reg}$ admits a complete K\"ahler
metric. For the line bundle $P$ we will take $P = L^k \otimes
K_M^{-1}$, so that an $(n,0)$-form valued in $P$ is simply a
section of $L^k$. For the metric on $P$ we take the metric induced by
the metric $h^k$ on $L^k$ whose curvature is $k\omega$, together
with the metric given by $\omega^n$ on $K_{M}$. The
curvature of $h_P$ then satisfies
\[ \sqrt{-1}F_{h_P} = k\omega + \mathrm{Ric}_{\omega} =
  (k+\lambda)\omega > \frac{1}{2}\omega_M, \]
for large enough $k$.

We will need the following $L^\infty$ and gradient estimates for
holomorphic sections of $L^k$.
\begin{prop}
  Let $f$ be a holomorphic section of $L^k$ over $M = X^{reg}$. We then
  have the following estimates
  \[ \sup_{M} |f|_{h^k} + |\nabla f|_{h^k, \omega_M} \leq
    K_1\Vert f\Vert_{L^2(M, h^k, \omega_M)}, \]
  where we emphasize that we are using the metrics $h^k$ and
  $\omega_M = k\omega$ to measure the various norms, and $K_1$ does not
  depend on $k$. 
\end{prop}
\begin{proof}
  Note first that $f$ extends to a holomorphic section of $L^k$ over
  $X$, using that $X$ is normal. Using 
  that $\omega$ has locally bounded potentials, we have that
  $\sup_X |f|_{h^k} < \infty$. 

  Next we show that $|\nabla f|_{h^k, \omega_M} < \infty$. For any $\hat{x}\in \hat{X}$, let $x =
  \hat{\Phi}_X(\hat x)\in X$.  We can find a section
  $s\in H^0(X, L)$ and some $r > 0$ such that $s(y) \not=0$ for $y\in
  B_{\eta_{FS}}(x,r)$. The assumption that $\omega$ is a K\"ahler
  current implies that we have constants
  $r' > 0$ and $C > 0$ (depending on $\hat{x}$) such that if we write
  $|s|^2_h = e^{-u}$, then $|u| < C$ on $X^{reg} \cap
  B_{\omega_M}(\hat{x}, r')$. We have $\Delta_{\omega_M} u = n$
  on $X^{reg}\cap B_{\omega_M}(\hat{x}, r'/2)$, and since $u$ is
  bounded, this equation extends to $B_{\omega_M}(\hat{x}, r'/2)$ by
  Lemma~\ref{lem:bddharmonic} and Lemma~\ref{lem:cutoff}. 
  The gradient estimate in Proposition~\ref{prop:RCDgrad} then implies that $|\nabla u| < C_1$ on
  $B_{\omega_M}(\hat{x}, r'/2)$. This implies that $|\nabla s| <
  C_2$ on $B_{\omega_M}(\hat{x}, r'/2)$. If $f$ is any holomorphic
  section of $L^k$, then on $B_{\omega_M}(\hat{x}, r'/2)$ the ratio
  $f/s^k$ is a bounded harmonic function, so using the gradient
  estimate again, together with the bounds for $s$, we find that
  $|\nabla f| < C_3$ on $B_{\omega_M}(\hat{x}, r'/4)$. We can cover
  $\hat{X}$ with finitely many ball of this type, showing that
  $|\nabla f|_{h^k, \omega_M} < \infty$ globally.

  We can obtain the effective estimates claimed in the proposition as
  follows. Since the curvature of $h^k$ is $\omega_M$, on $M$ we have
  \[\label{eq:Deltaf2} \Delta_{\omega_M} |f|^2_{h^k} = |\nabla f|^2_{h^k,
     \omega_M} - n |f|^2_{h^k}. \]
  Let $\phi_i$ denote cutoff functions as in
  Lemma~\ref{lem:cutoff}. We have, omitting the subscripts, 
  \[ \int_{M} \phi_i^2 |\nabla f|^2 \omega_M^n &=
    \int_{M} \phi_i^2 (\Delta |f|^2 + n |f|^2)\, \omega_M^n
    \\
    &= \int_{M} (-4\phi_i |f| \nabla \phi_i\cdot \nabla |f| +
    \phi_i^2n |f|^2)\, \omega_M^n \\
    &\leq \int_{M} \left( \frac{1}{2} \phi_i^2 |\nabla f|^2 + 8
        |\nabla\phi_i|^2 |f|^2 + \phi_i^2 n |f|^2\right)
      \,\omega_M^n. 
    \]
    Letting $i\to\infty$, and using that $|f|\in L^\infty$, we get
    \[ \int_{M} |\nabla f|^2\, \omega_M^n \leq
      2n\int_{M} |f|^2\, \omega_M^n. \]

    We also have the following Bochner type formula on $M$ (see e.g. La
    Nave-Tian-Zhang~\cite[Lemma 3.1]{LNTZ}):
    \[ \label{eq:Deltadf2} \Delta |\nabla f|^2 \geq \mathrm{Ric}_{\omega_M}(\nabla f,
      \nabla f) - (n+2) |\nabla f|^2 \geq -(n+2+|\lambda|) |\nabla f|^2, \]
    where we are using the metrics $h^k, \omega_M$ as above.

    Both \eqref{eq:Deltaf2} and \eqref{eq:Deltadf2} are of the form
    \[ \label{eq:Dvineq} \Delta v \geq - Av, \]
    where $v$ is a smooth $L^\infty$ function on $M$. We can
    argue using the cutoff functions $\phi_i$, as in the proof of
    Lemma~\ref{lem:bddharmonic}, to show that
    $v$ satisfies this differential inequality on all of $\hat{X}$ in
    a weak sense, i.e. for any Lipschitz test function $\rho \geq 0$ we have
    \[ \int_{M} (-\nabla \rho\cdot \nabla v + A\rho v)\, \omega_M^n \geq
      0. \]
    Using this, together with estimates for the heat kernel on
    $\hat{X}$, we can obtain the required $L^\infty$ bound for $v =
    |f|^2$ and $v = |\nabla f|^2$. More precisely, using \cite[Theorem
    1.2]{JLZ}, together with the RCD property in
    Proposition~\ref{prop:RCD},  we obtain an $L^2$-bound for the heat kernel
    $H(x,y,1)$ on $M$, independently of $k$. Using
    \eqref{eq:Dvineq}, for any $x\in M$ we have
    \[ \frac{d}{dt} \int_M v(y) H(x, y,t)\, \omega_M^n(y) =
      \int_M v(y) \Delta_y H(x,y,t)\, \omega_M^n(y) \geq -A \int_M
      v(y)H(x,y,t)\, \omega_M^n(y),  \]
    so
    \[ v(x) \leq e^A \int_M v(y) H(x,y,1)\, \omega_M^n(y) \leq e^A
      C \Vert v\Vert_{L^2}, \]
    as required. 
\end{proof}

In order to show that sections of $L^k$ separate points of
$\hat{X}$ for large $k$ (and therefore also
for $k=1$), we follow
the approach of Donaldson-Sun~\cite{DS1}, constructing suitable
sections of $L^k$ using H\"ormander's $L^2$-estimate. For 
this the basic ingredient in \cite{DS1} is to consider a tangent
cone $Z$ of $\hat{X}$ at $x$, and use that the regular part of $Z$
is a K\"ahler cone, while at the same time the singular set can be
excised by a suitable cutoff function. The main new difficulty in our
setting is that along the pointed convergence of a sequence of
rescalings
\[ \label{eq:tanconelimit} (\hat{X}, \lambda_i d_{\hat{X}}, x) \to (Z,
  d_Z, o),\]
with $\lambda_i\to \infty$ we do
not know that compact subsets $K\subset Z^{reg}$ of the (metric)
regular set in $Z$ are obtained as smooth limits of subsets of the (complex
analytic) regular set $X^{reg}$. For example,  a priori it
may happen that along
the convergence in \eqref{eq:tanconelimit}, even if
$Z=\mathbb{R}^{2n}$, the singular set
$X\setminus X^{reg}$ converges to a dense subset of $Z$. This
is similar to the issue dealt with in Chen-Donaldson-Sun~\cite{CDS3},
but in that work it is used crucially that the singular spaces
considered are limits of smooth manifolds with lower Ricci bounds. 

To deal with this issue in our setting,  we  exploit the fact  that
$X\setminus X^{reg}$ is locally 
    contained in the zero set of holomorphic functions, which also
    define harmonic functions on the RCD space $\hat{X}$.  Crucially,
    these functions have a bound on their order of vanishing
    (Lemma~\ref{lem:finitevanish}), which can be used to control the
    size of the zero set at different scales, at least on balls that
    are sufficiently close to a Euclidean ball. This can be used to
    show that balls in $\hat{X}$ that are almost Euclidean are
    contained in $X^{reg}$ (Proposition~\ref{prop:almostreg31}). This is the main new
    ingredient in our argument. 
      Given this, we can closely follow the arguments in
      Donaldson-Sun~\cite{DS1} or \cite{LSz1} to
    construct holomorphic sections of $L^k$.

Let us write $\Gamma = X\setminus X^{reg}$ for the algebraic singular set. 
Observe that $\Gamma$ can locally
be cut out by holomorphic functions. Therefore, we can cover $X$
with open sets $U_k'$ and we have nonzero holomorphic functions $s_k$
on $U_k'$ such that $\Gamma\cap U_k' \subset
s_k^{-1}(0)$. We can assume that the $s_k$ are bounded, and that we
have relatively compact open sets $U_k \ssubset U_k'$ that still cover
$X$. We let $\hat{U}_k, \hat{U}_k'$ be the corresponding open
sets pulled back to $\hat{X}$. Using Lemma~\ref{lem:cutoff}, we can 
extend the $s_k$ to complex valued harmonic functions on
$\hat{X}$, which vanish along $\Gamma$. Our first
task will be to show that we have a bound for the order of vanishing
of the $s_k$ at each point. Note first that by the assumption that
$\omega$ is a K\"ahler current, there exists an $r_0 > 0$ such that if
$p\in \hat{U}_k$, then $B(p, r_0) \subset \hat{U}_k'$. Here,
and below, a ball $B(p,r)$ always denotes the metric ball using the
metric $d_{\hat X}$ on $\hat{X}$ induced by $d_{\omega}$ on
$X^{reg}$. 

\begin{lemma}\label{lem:finitevanish}
  There are constant $c_1, N > 0$, depending on $(X, \omega)$, such
  that for any $\hat{x} \in \hat{U}_k$ and $r\in (0,r_0)$, we have
  \[ \int_{B(\hat{x}, r)} |s_k|^2\, \omega^n \geq c_1 r^N,\]
  for all $r < r_0$.
\end{lemma}
\begin{proof}
 First note that since $\hat{X}$ is a non-collapsed RCD space, we
  have a constant $\nu > 0$ such that $\mathrm{vol}\, B(\hat{x}, r) >
  \nu r^{2n}$ for all $r < 1$. At the same time we can bound the
  volume of sublevel sets $U_k'\cap \{|s_k| < t\}$ from above, using
  the assumptions on $\omega$. Indeed, on $U_k'$ we have $\omega^n =
  F\eta^n_{FS}$, and $F\in L^p(X, \eta^n_{FS})$ for some $p > 1$. 
It follows that for any $t > 0$ we have
  \[ \mathrm{vol} (U_k'\cap \{ |s_k| < t\}, \omega^n) &= \int_{U_k'\cap \{|s_k| < t\}}
    \omega^n \\
    &= \int_{U_k'\cap \{|s_k| < t\}} F\, \eta_{FS}^n \\
    &\leq C_1 \mathrm{vol}(U_k'\cap \{|s_k| < t\}, \eta_{FS}^n)^{1/p'}
    \left(\int_{U_k'} F^p\, \eta_{FS}^n\right)^{1/p} \\
    &\leq C_2  \mathrm{vol}(U_k'\cap \{|s_k| < t\}, \eta_{FS}^n)^{1/p'}, \]
  for suitable constants $C_1, C_2$ independent of $t$, and $p'$ is the
  conjugate exponent of $p$. Since $|s_k|^{-\epsilon}\, \eta_{FS}^n$ is integrable for
  some $\epsilon > 0$, it follows that we have a bound
  \[ \mathrm{vol}(U_k'\cap \{|s_k| < t\}, \eta_{FS}^n) \leq C_3 t^\epsilon, \]
  and so in sum we have
  \[ \mathrm{vol} (U_k'\cap \{ |s_k| < t\}, \omega^n)  \leq C_4
    t^{\alpha},\]
  for some $C_4, \alpha > 0$ independent of $t$. Given a small $r > 0$
  such that $B(\hat{x}, r)\subset \hat{U}_k'$, choose $t_r$ such that
  \[ C_4 t_r^\alpha = \frac{1}{2} \nu r^{2n}, \]
  i.e.
  \[ t_r = \left(\frac{\nu}{2C_4}\right)^{1/\alpha} r^{2n/\alpha} =
  c_5 r^{2n\alpha^{-1}}, \]
  for suitable $c_5> 0$. By our estimates for the volumes, we
  then have
  \[ \mathrm{vol}( B(\hat{x}, r) \cap \{|\hat{s}_k| \geq t_r\}) \geq
    \frac{1}{2}\nu r^{2n}, \]
  and so
  \[ \label{eq:L2lower}
    \int_{B(\hat{x}, r)} |\hat{s}_k|^2\, \omega^n &\geq  \frac{c_5^2
      r^{4n\alpha^{-1}}}{2} \nu r^{2n} = c_1 r^N,\]
  for some $c_1, N > 0$, independent of $r$, as required.
\end{proof}

Next we need a version of the three annulus lemma for almost Euclidean
balls, similar to \cite[Theorem 0.7]{Ding}.
\begin{lemma}\label{lem:3ann}
  For any $\mu > 0$, $\mu\not\in\mathbb{Z}$, there is an $\epsilon >
  0$ depending on $\mu, n$ with the following property. Suppose that
  $B(p,1)$ is a unit ball in a noncollapsed $RCD(-1,2n)$-space such
  that
  \[ d_{GH}(B(p,1), B(0_{\mathbb{R}^{2n}},1)) < \epsilon, \]
  where $0_{\mathbb{R}^{2n}}$ denotes the origin in Euclidean space. 
  Let $u : B(p,1)\to\mathbb{C}$ be a harmonic function such that
  \[ \left(\fint_{B(p,1/2)} |u|^2 \right)^{1/2} \geq 2^\mu
    \left(\fint_{B(p,1/4)} |u|^2\right)^{1/2}. \]
  Then
  \[ \left(\fint_{B(p,1)} |u|^2 \right)^{1/2} \geq 2^\mu \left(\fint_{B(p,1/2)}
      |u|^2\right)^{1/2}. \]
\end{lemma}
\begin{proof}
  The proof is by contradiction, similarly to \cite{Ding},
  based on the fact that on the
  Euclidean space $\mathbb{R}^{2n}$ every homogeneous harmonic
  function has integer degree.
\end{proof}

Combining the previous two results, we have the following, controlling
the decay rate of the defining functions $\hat{s}_k$ around almost
regular points. 
\begin{lemma}\label{lem:growthbound30}
  There exists an $\epsilon_0, r_0 > 0$, depending on $(X,
  \omega)$, such that if $\hat{x}\in \hat{U}_k$ and for some
  $r_1\in (0,r_0)$ we have 
  \[ \label{eq:dGH1} d_{GH}(B(\hat{x}, r_1), B(0_{\mathbb{R}^{2n}}, r_1)) <
    r_1\epsilon_0, \]
  then
  \[ \label{eq:growthbound21} \limsup_{r\to 0} \frac{\fint_{B(\hat{x},
        r)} |\hat{s}_k|^2\, \omega^n}{\fint_{B(\hat{x}, r/2)} |\hat{s}_k|^2\, \omega^n} \leq
    2^{2N}, \]
  for the $N$ in Lemma~\ref{lem:finitevanish}. 
\end{lemma}
\begin{proof}
  Fix $\mu \in (N/2, N)$ such that $\mu\not\in \mathbb{Z}$. If $\epsilon_0$ and $r_0$
  are sufficiently small (depending on $\mu$),
  then the inequality \eqref{eq:dGH1} implies
  that for any $r \leq r_1$ we have
  \[ d_{GH}(B(\hat{x}, r), B(0_{\mathbb{R}^{2n}}, r)) <
    r\epsilon, \]
  for the $\epsilon$ in Lemma~\ref{lem:3ann}, and so the conclusion of
  that Lemma holds. It follows that if 
  \[\label{eq:a10} \left(\fint_{B(\hat{x}, r)} |\hat{s}_k|^2\, \omega^n\right)^{1/2} \geq
    2^\mu \left( \fint_{B(\hat{x}, r/2)} |\hat{s}_k|^2\, \omega^n
    \right)^{1/2}, \]
  for some $r\leq r_1$, then applying Lemma~\ref{lem:3ann} inductively, we have
  \[ \left(\fint_{B(\hat{x}, 2^j r)} |\hat{s}_k|^2\, \omega^n \right)^{1/2} \geq
    2^{j\mu} \left( \fint_{B(\hat{x}, r/2)} |\hat{s}_k|^2\, \omega^n
    \right)^{1/2}, \]
  as long as $2^jr \leq r_1$. Given any $r \leq r_1$, if we let $\bar
  j$ denote the largest $j$ such that $2^jr\leq r_1$, then we obtain
  \[ \left(\fint_{B(\hat{x}, r/2)} |\hat{s}_k|^2\, \omega^n\right)^{1/2} \leq
    2^{-\bar j\mu}C, \]
  where $C$ is independent of $r$, but depends on the $L^2$-norm of
  $\hat{s}_k$ on $B(\hat{x}, r_1)$. Applying
  Lemma~\ref{lem:finitevanish}
  we then have
  \[ 2^{-\bar j \mu} C\geq c_1^{1/2} (r/2)^{N/2}. \]
  Since $2^{\bar j+1} r > r_1$, it follows that $2^{-\bar j \mu} <
  (2r/r_1)^\mu$, so
  \[ c_1^{1/2}(r/2)^{N/2} < (2r/r_1)^\mu C. \]
  Since $\mu > N/2$, this inequality implies a lower bound
  for $r$ satisfying \eqref{eq:a10}. The required conclusion
  \eqref{eq:growthbound21} follows. 
\end{proof}

Using this result, we will show that almost Euclidean balls are
contained in the complex analytically regular set $X^{reg} \subset
\hat{X}$. Note that the assumption \eqref{eq:growthest41} will hold on
sufficiently small balls around a given point, by the previous lemma. 

\begin{prop}\label{prop:almostreg1}
  There exists an $\epsilon_2 > 0$, depending on $(X, \omega)$, with
  the following property. Suppose that $\hat{x}\in \hat{U}_j$ and $k >
  0$ is a large integer such that $\epsilon_2^{-1}k^{-2} <
  \epsilon_2$. Suppose in addition that
  \[ d_{GH}\Big( B(\hat{x}, \epsilon_2^{-1}k^{-2}),
    B_{\mathbb{R}^{2n}}(0, \epsilon_2^{-1}k^{-2})\Big) < \epsilon_2
    k^{-2}, \]
  and that
  \[ \label{eq:growthest41} \fint_{B(\hat{x}, \epsilon_2^{-1}k^{-2})} |\hat{s}_j|^2\, \omega^n \leq
    2^{2N} \fint_{B(\hat{x}, \frac{1}{2}\epsilon_2^{-1}k^{-2})}
    |\hat{s}_j|^2\, \omega^n, \]
  for the $N$ in Lemma~\ref{lem:finitevanish}. Then $\hat{x}\in
  X^{reg}$, where $X^{reg}$ is the complex analytically regular set of
  $X$, viewed as a subset of $\hat{X}$.
\end{prop}
\begin{proof}
 We will argue by contradiction, similarly to
  \cite[Proposition 3.1]{LSz1} which in turn is based on
  Donaldson-Sun~\cite{DS1}. Suppose that no suitable $\epsilon_2$
  exists. Then we have a sequence of points $\hat{x}_i$, and
  integers $k_i > i$ such that the hypotheses are satisfied (with
  $\epsilon_2 = 1/i$). We will show that for sufficiently large $i$
  we have $\hat{x}_i \in X^{reg}$ by constructing holomorphic
  coordinates in a neighborhood of $\hat{x}_i$. 

  By a slight abuse of
  notation we will write $\hat{U}_i, \hat{s}_i$ instead of
  $\hat{U}_{j_i}$ and $\hat{s}_{j_i}$ to simplify the notation. The
  assumptions imply that the rescaled balls
  \[ \label{eq:conv10} k_i^{1/2} B(\hat{x}_i, i k_i^{-1/2}) \to \mathbb{R}^{2n}, \]
  in the pointed Gromov-Hausdorff sense. Using
  Lemma~\ref{lem:3ann} together with the condition
  \eqref{eq:growthest41},
  we can extract a nontrivial
  limit of the normalized functions
  \[ \tilde{s}_i = \frac{\hat{s}_i}{\fint_{B(\hat{x}_i, k_i^{-1/2})}
      |\hat{s}_i|^2\, \omega^n}. \]
  Indeed, we have
  \[ \label{eq:L2norm10} \fint_{B(\hat{x}_i, k_i^{-1/2})} |\hat{s}_i|^2\,
  \omega^n = 1, \] and using Lemma~\ref{lem:3ann} with some $\mu\in (N,
  2N)$, together with
  \eqref{eq:growthest41}, implies that for sufficiently large $i$ we
  have 
  \[ \fint_{B(\hat{x}_i, 2^{-j} \epsilon_2^{-1} k_i^{-1/2})} |\hat{s}_i|^2\,\omega^n \leq
    2^{4N}\fint_{B(\hat{x}_i, 2^{-j-1} \epsilon_2^{-1}k_i^{-1/2})}
    |\hat{s}_i|^2\,\omega^n, \]
  for all $j \geq 0$. In particular, viewed as functions on the
  rescaled balls $k_i^{1/2} B(\hat{x}_i, ik_i^{-1/2})$, the $L^2$
  norms of the $\hat{s}_i$ are bounded independently of $i$ on any
  $R$-ball. Using the gradient estimate,
  Proposition~\ref{prop:RCDgrad}, it follows that up to choosing a
  subsequence, the functions $\tilde{s}_i$ converge locally uniformly
  to a harmonic function $\tilde{s}_\infty :
  \mathbb{R}^{2n}\to\mathbb{C}$. As a consequence, $\tilde{s}_\infty$ is 
  smooth, and because of the normalization
  \eqref{eq:L2norm10}, $\tilde{s}_\infty$ is nonzero.
  
  Note that if we take a sequence of rescalings of $\mathbb{R}^{2n}$
  with factors going to infinity, and consider the corresponding
  pullbacks of $\tilde{s}_\infty$, normalized to have unit $L^2$-norm
  on the unit balls, then this new sequence of harmonic functions will
  converge to the leading order homogeneous piece of the Taylor
  expansion of $\tilde{s}_\infty$ at the origin (up to a constant
  factor). This means that in the procedure above, up to replacing the
  integers $k_i$ by suitable larger integers, we can
  assume that the limit $\tilde{s}_\infty$ is in fact homogeneous. 

  Let us write $\Sigma =
  \tilde{s}_\infty^{-1}(0)$. Our next goal is to show that under the
  convergence in \eqref{eq:conv10}, the set $\mathbb{R}^{2n}\setminus
  \Sigma$ is the locally smooth limit of subsets of $X^{reg}$, and
  that $\tilde{s}_\infty$ is actually a holomorphic function under an
  identification $\mathbb{R}^{2n} = \mathbb{C}^n$. Then we will be
  able to follow the argument in the proof of \cite[Proposition
  3.1]{LSz1} with the cone $V = \mathbb{C}^n$, but treating
  $\Sigma$ as the singular set.

  Note that since $\tilde{s}_\infty$ is a nonzero harmonic function,
  the set $\mathbb{R}^{2n}\setminus \Sigma$ is open and dense in
  $\mathbb{R}^{2n}$. 
  Suppose that $V\subset \mathbb{R}^{2n}$ is an open
  subset such that $\bar{V}$ is compact and $|\tilde{s}_\infty| > 0$
  on $\bar{V}$. Then, because of the local uniform convergence of
  $\tilde{s}_i$ to $\tilde{s}_\infty$, and the fact that the sets
  $\tilde{s}_i\not=0$ are contained in $X^{reg}$, it follows that we
  have open subsets $V_i \subset\subset k_i^{-2}B(\hat{x}_i,
  ik_i^{-1/2})\cap X^{reg}$, which converge in the Gromov-Hausdorff
  sense to $V$. The metrics on the $V_i$ are smooth non-collapsed
  K\"ahler-Einstein metrics, so using Anderson's $\epsilon$-regularity
  result~\cite{And}, up to choosing a
  subsequence, the complex structures on $V_i$ converge to a complex
  structure on $V$ with respect to which the Euclidean metric is
  K\"ahler. Note that we do not yet know that
  $\mathbb{R}^{2n}\setminus \Sigma$ is connected, and in principle we
  may get different complex structures on different connected
  components. Our next goal is to show that the Hausdorff dimension of
  $\Sigma$ is at most $2n-2$, which will show that the complement of
  $\Sigma$ is connected. 

  We can assume that the holomorphic functions
  $\tilde{s}_i$ on $V_i$ converge to a holomorphic function $\tilde{s}_\infty$
  on $V$. Writing
  $\tilde{s}_\infty = u_\infty + \sqrt{-1} v_\infty$, we therefore
  have $\langle \nabla
  u_\infty, \nabla v_\infty\rangle = 0$ and $|\nabla v_\infty| =
  |\nabla u_\infty|$ on $\mathbb{R}^{2n}\setminus \Sigma$, and by
  density these relations extend to all of $\mathbb{R}^{2n}$. We can assume that $u_\infty$ is
  non-constant. Let $\alpha > 2n-2$, and suppose that the Hausdorff measure
  $\mathcal{H}^\alpha(\Sigma) > 0$. By Caffarelli-Friedman~\cite{CF}
  (see also Han-Lin~\cite{HL}) we know that
  $\mathcal{H}^\alpha(\Sigma \cap |\nabla u_\infty|^{-1}(0)) =0$, and
  so we can find an $\alpha$-dimensional point of density $q$ of
  $\Sigma \setminus |\nabla u_\infty|^{-1}(0)$. Since
  $\nabla u_\infty(q)\not=0$, it follows that $\nabla
  v_\infty(q)\not=0$ and $\langle \nabla v_\infty(q), \nabla
  u_\infty(q)\rangle = 0$. Therefore in a neighborhood of $q$ the set
  $\Sigma$ is a smooth $2n-2$-dimensional submanifold, contradicting
  that $q$ is an $\alpha$-dimensional point of density. In conclusion
  $\dim_{\mathcal{H}} \Sigma \leq 2n-2$, and so
  $\mathbb{R}^{2n-2}\setminus \Sigma$ is connected.

  We can therefore assume that in the argument above the complex
  structure that we obtain on $\mathbb{R}^{2n}\setminus \Sigma$ agrees with the
  standard structure on $\mathbb{C}^n$, and $\tilde{s}_\infty$ is a
  holomorphic function on $\mathbb{C}^n\setminus \Sigma$, but since
  it is smooth, it is actually holomorphic on $\mathbb{C}^n$.
  In particular $\tilde{s}_\infty^{-1}(0)$ is a complex hypersurface
  defined by a homogeneous holomorphic function. 

  At this point we can closely follow the  proof of \cite[Proposition
  3.1]{LSz1}), treating the zero set $\tilde{s}_\infty^{-1}(0)$ as the
  singular set $\Sigma$ in \cite{LSz1}. The properties of the set
  $\Sigma$ that are used are that the tubular $\rho$-neighborhood
  $\Sigma_\rho$ satisfies the volume bounds $\mathrm{vol}(\Sigma_\rho\cap
  B(0,R))\leq C_R \rho^2$, where the constant $C_R$ in our setting could
  depend on $R, \tilde{s}_\infty$. In addition if $B(p, 2r)
  \in \mathbb{R}^{2n}\setminus \Sigma$, then $B(p, r)$ is the Gromov-Hausdorff limit
  of balls $B(p_i, r) \subset (M, k_i\omega)$ in K\"ahler-Einstein
  manifolds, and so by  Anderson's result~\cite{And} we have good
  holomorphic charts on the $B(p_i, r)$ for sufficiently large $i$,
  analogous to those in \cite[Theorem 1.4]{LSz1}. 
   The rest of the proof is then identical to
  the argument in the proof of \cite[Proposition 3.1]{LSz1} (see also
  Donaldson-Sun~\cite{DS1}) to show that for sufficiently large $i$ we can construct holomorphic
  sections $s_0, \ldots, s_n$ of $L^{k_i'}$ for suitable powers
  $k_i'$, such that $\frac{s_1}{s_0}, \ldots, \frac{s_n}{s_0}$ define
  a generically one-to-one map from a neighborhood of $x_i =
  \hat{\Phi}_X(\hat{x}_i)$ in $X$ to a subset of $\mathbb{C}^n$. Since
  $X$ is normal, it follows that the map is one-to-one, and so $x_i\in
  X^{reg}$. Therefore $\hat{x}_i\in X^{reg}$ as claimed. 
\end{proof}

For any $\epsilon > 0$, let us define the $\epsilon$-regular set
$\mathcal{R}_\epsilon(Y)$ in a noncollapsed RCD space $Y$  to be the set of points $p$ that
satisfy
\[ \label{eq:epsreg}
  \lim_{r\to 0} r^{-2n} \mathrm{vol}(B(p,r)) > \omega_{2n} -
  \epsilon, \]
where $\omega_{2n}$ is the volume of the $2n$-dimensional Euclidean
unit ball. Then $\mathcal{R}_\epsilon(Y)$ is an open set, and
from the previous result we obtain the following.

\begin{prop}\label{prop:almostreg31}
  There exists an $\epsilon_3 > 0$, depending on $(X, \omega)$, such
  that the $\epsilon_3$-regular set $\mathcal{R}_{\epsilon_3}(\hat{X})\subset
  \hat{X}$ coincides with the complex analytically regular set
  $X^{reg}$.  
\end{prop}
\begin{proof}
  It is clear that $X^{reg}\subset
  \mathcal{R}_{\epsilon_3}(\hat{X})$. To see the reverse inclusion,
  note that by Cheeger-Colding~\cite{CC1}, and De Philippis-Gigli~\cite{DPG} in
  the setting of non-collapsed RCD spaces, given the $\epsilon_2 > 0$ in
  Proposition~\ref{prop:almostreg1}, there exists an $\epsilon_3 > 0$
  such that if $\hat{x} \in \mathcal{R}_{\epsilon_3}$, then for all
  sufficiently large $k$ (depending on $\hat{x}$), we have
  \[ d_{GH}\Big( B(\hat{x}, \epsilon_2^{-1}k^{-2}),
    B_{\mathbb{R}^{2n}}(0, \epsilon_2^{-1}k^{-2})\Big) < \epsilon_2
    k^{-2}. \]
  Using also Lemma~\ref{lem:growthbound30} (and choosing $\epsilon_3$
  smaller if necessary), we have the growth estimate
  \eqref{eq:growthest41}. Proposition~\ref{prop:almostreg1} then
  implies that $\hat{x}\in X^{reg}$. 
\end{proof}

This has the following immediate corollary.
\begin{cor}\label{cor:regset}
  There is an $\epsilon > 0$, depending on $(X, \omega)$, such that
  the $\epsilon$-regular set $\mathcal{R}_\epsilon(\hat{X})$ coincides
  with the metric regular set of $\hat{X}$, i.e. the points
  $\hat{x}\in \hat{X}$ where the tangent cone is $\mathbb{R}^{2n}$. 
\end{cor}

Given these preliminaries, we have the following result, analogous to \cite[Proposition
3.1]{LSz1} in our setting. 

\begin{prop}\label{prop:goodsection}
  Let $(V,o)$ be a metric cone, such that for any $\epsilon > 0$ the
  singular set $V\setminus \mathcal{R}_\epsilon(V)$ has zero capacity
  (in the sense of (3) in Definition~\ref{defn:almostsmooth}). Let $\zeta > 0$.
  There are $K, \epsilon, C > 0$, depending on
  $\zeta, (X, \omega), V$ satisfying the following property. Suppose
  that $k$ is a large integer such that $\epsilon^{-1}k^{-1/2} <
  \epsilon$ and for some $\hat{x}\in \hat{X}$
  \[ d_{GH}\Big(B(\hat{x}, \epsilon^{-1}k^{-1/2}), B(o,
    \epsilon^{-1} k^{-1/2})\Big) < \epsilon k^{-1/2}.\]
  Then for some $m < K$ the line bundle $L^{mk}$ admits a holomorphic
  section $s$ over $M=X^{reg}\setminus D$ such that $\Vert s\Vert_{L^2(h^{mk}, mk \omega)} < C$
  and
  \[ \label{eq:sbound20}\Big| |s(z)| - e^{-mk d(z, \hat{x})^2/2}\Big| < \zeta \]
  for $z\in M$. 
\end{prop}

  Given the results above, the argument is essentially the same as
  that in \cite{LSz1} (see also Donaldson-Sun~\cite{DS1}). One main difference is that in the 
  setting of noncollapsed RCD spaces the sharp estimates of
  Cheeger-Jiang-Naber~\cite{CJN} do not yet seem to be available in
  the literature. However, the proof of \cite[Proposition 3.1]{LSz1}
  applies under the assumption that for any $\epsilon > 0$ the singular set $\Sigma =
  V\setminus \mathcal{R}_{\epsilon}(V)$ has zero capacity.

  We can rule out non-flat (iterated) tangent cones that split
  off a Euclidean factor of $\mathbb{R}^{2n-2}$, following the approach
  of Chen-Donaldson-Sun~\cite[Proposition 12]{CDS2} (see also
  \cite[Proposition 3.2]{LSz1}).

  \begin{prop}\label{prop:nocodim2}
    Suppose that $\hat{x}_j\in \hat{X}$ and for a sequence of integers
    $k_j\to\infty$ the rescaled pointed sequence $(\hat{X}, k_j^2
    d_{\hat{X}}, \hat{x}_j)$ converges to $\mathbb{R}^{2n-2}\times
    C(S^1_\gamma)$ in the pointed Gromov-Hausdorff sense. Here $C(S^1_\gamma)$ is the
  cone over a circle of length $\gamma$. Then $\gamma
    = 2\pi$, i.e. $C(S^1_\gamma) = \mathbb{R}^2$. 
  \end{prop}
  \begin{proof}
  If $V =\mathbb{R}^{2n-2}\times C(S^1_\gamma)$, then the singular set of $V$
  has capacity zero, and so Proposition~\ref{prop:goodsection} can be
  applied. Then, as in \cite[Proposition 12]{CDS2}, it follows that
  for sufficiently large $j$, we can find a biholomorphism $F_j$ from a
  neighborhood $\Omega_j$ of $\hat{x}_j$ to the unit ball
  $B(0,1)\subset \mathbb{C}^n$. In particular $B(\hat{x}_j,
  \frac{1}{2}k_j^{-2}) \subset X^{reg}$,
  and then the limit $\mathbb{R}^{2n-2}\times C(S^1_\gamma)$ of
  $(\hat{X}, k_j^2d_{\hat{X}}, \hat{x}_j)$ must be smooth at the
  origin. Therefore $\gamma=2\pi$. 
\end{proof}

As a consequence of this result we can prove
Theorem~\ref{thm:homeomorphic}. 
\begin{proof}[Proof of Theorem~\ref{thm:homeomorphic}]
Using Propositions~\ref{prop:nocodim1} and \ref{prop:nocodim2}, and De
Philippis-Gigli's dimension estimate~\cite{DPG} for the singular set
(extending Cheeger-Colding~\cite{CC1}), it follows that the singular
set of any iterated tangent cone of $\hat{X}$ has Hausdorff codimension at least
3. Using Proposition~\ref{prop:almostreg31} we know that the singular
set is closed, and so as in Donaldson-Sun~\cite[Proposition 3.5]{DS1}
we see that the singular set of any iterated tangent cone has capacity zero. In particular
Proposition~\ref{prop:goodsection} can be applied to any $(V,o)$ that
arises as a rescaled limit of $\hat{X}$. 

Suppose that $p\not= q$ are points in $\hat{X}$. Applying
Proposition~\ref{prop:goodsection} to tangent
cones at $p, q$, we can find sections $s_p$ and $s_q$
of some powers $L^{m_p}, L^{m_q}$, such that $|s_p(p)| > |s_p(q)|$,
and $|s_q(q)| > |s_q(p)|$. Taking powers we find that the sections
$s_p^{m_q}$ and $s_q^{m_p}$ of $L^{m_pm_q}$ separate the points $p,q$,
and so the map $\hat{\Phi}_X$ is injective as required. 
\end{proof}

To complete the proofs of Theorem~\ref{thm:main}, it remains to show the
codimension bounds for the singular
set of $\hat{X}$. By the dimension estimate of \cite{DPG}, it suffices
to show the following. Note that this result would follow from a
version of Cheeger-Colding-Tian~\cite[Theorem 9.1]{CCT} for RCD spaces, but in our
setting we can give a more direct proof. 
\begin{prop}
  In the setting of Theorem~\ref{thm:main}, 
  suppose that a tangent cone $\hat{X}_p$ at $p\in \hat{X}$ splits off
  an isometric factor of $\mathbb{R}^{2n-3}$. Then $\hat{X}_p =
  \mathbb{R}^{2n}$. In particular in the stratification of the
  singular set of $\hat{X}$ we have $S_{2n-1} = S_{2n-4}$, and so
  $\dim_{\mathcal{H}} S \leq 2n-4$. 
\end{prop}
\begin{proof}
  Suppose that $\hat{X}$ has a tangent cone of the form $\hat{X}_p
  = C(Z)\times \mathbb{R}^{2n-3}$, where $Z$ is
  two-dimensional. If $Z$ had a singular point, necessarily with
  tangent cone $C(S^1_\gamma)$ for some $\gamma < 2\pi$, then
  $\hat{X}$ would have an iterated tangent cone of the form
  $\mathbb{R}^{2n-2}\times C(S^1_\gamma)$. This is ruled out by
  Proposition~\ref{prop:nocodim2}. Therefore $Z$ is
  actually a smooth two dimensional Einstein manifold with metric satisfying
  $\mathrm{Ric}(h) = h$. This implies that $Z$ is the unit 2-sphere,
  and it follows that $\hat{X}_p = \mathbb{R}^{2n}$ so that $p$ is a
  regular point. Therefore the singular set of $\hat{X}$ coincides
  with $S_{2n-4}$, as required. 
\end{proof}

\section{CscK approximations}\label{sec:cscKapprox}
In this section we will prove Theorem~\ref{thm:cscKapprox}. Thus, let
$(X, \omega_{KE})$ be an $n$-dimensional singular K\"ahler-Einstein space, such that
the automorphism group of $X$ is discrete and $\omega_{KE} \in c_1(L)$
for an ample $\mathbb{Q}$-line bundle on $X$. On the regular part we have
$\mathrm{Ric}(\omega_{KE}) = \lambda \omega_{KE}$ for a constant
$\lambda\in \mathbb{R}$. We will assume that $\lambda\in
\{0,-1,1\}$. In the latter two cases we have $L = \pm K_X$. We first
recall the properness  of the Mabuchi K-energy in this singular
setting. This has been well studied in the Fano setting (see
Darvas~\cite{DarvasMetric} for example), but we were
not able to find the corresponding much easier result in the literature for
singular varieties in the case when $\lambda \leq 0$.

First recall the definitions of certain functionals (see
Darvas~\cite{DarvasMetric} or
Boucksom-Eyssidieux-Guedj-Zeriahi~\cite{BEGZ} for instance). We choose a
smooth representative $\omega\in c_1(L)$. This means that
$m\omega$ is the pullback of the Fubini-Study metric under an
embedding using sections of $L^m$ for large $m$. In general we define
a function $f : U\to \mathbb{R}$ on an open set $U\subset X$ to be
smooth, and write $f\in C^\infty(U)$, if it is the restriction of a smooth function under an
embedding $U\subset \mathbb{C}^N$. We let
\[ \mathcal{H}_{\omega}(X) &= \{ u\in C^\infty(X)\,:\, \omega_u :=
  \omega + \ddbar u > 0\}, \\
  PSH_{\omega}(X) &= \{ u\in L^1(X)\,:\, \omega_u:= \omega + \ddbar
  u \geq 0\}. \]

We define the $\mathcal{J}_\omega$ functional on $PSH_\omega(X)\cap
L^\infty$ by setting $\mathcal{J}_\omega(0)=0$ and the variation
\[ \delta\mathcal{J}_\omega(u) = n\int_{X^{reg}} \delta u (\omega -
  \omega_u)\wedge \omega_u^{n-1}. \]

Let us choose a smooth metric $h$ on $K_X$, i.e. if $\sigma$ is a
local non-vanishing section of $K_X^r$, then the norm $|\sigma|^2_{h^r}$ is
a smooth function. The adapted measure $\mu$ is defined
using such local trivializing sections 
to be (see \cite[Section 6.2]{EGZ})
\[ \mu = (i^{rn^2} \sigma\wedge \bar\sigma)^{1/r}
  |\sigma|_{h^r}^{-2/r} \text{ on } X^{reg}, \]
extended trivially to $X$. Recall that if $X$ has klt singularities,
then $\mu$ has finite total mass. Moreover, if $\pi: Y \to X$ is a
resolution, and $\Omega$ is a smooth volume form on $Y$, then we have
\[ \label{eq:FLp} \pi^*\mu = F \Omega\, \text{ on } \pi^{-1}(X^{reg}), \]
where $F \in L^p(\Omega)$ for some $p > 1$ (see \cite[Lemma
6.4]{EGZ}).
In our three cases $\lambda\in \{0,-1,1\}$ we can choose the metric
$h$ in such a  way that the curvature of $h$ is given by
$-\lambda\omega$ for the smooth metric $\omega$. 

We define the Mabuchi K-energy, for $u\in PSH_\omega(X)\cap L^\infty$, by
\[ M_\omega(u) = \int_{X^{reg}}
  \log\left(\frac{\omega_u^n}{\mu}\right) \omega_u^n -\lambda
  \mathcal{J}_\omega(u). \] 
The first term (the entropy) is defined to be $\infty$, unless
$\omega_u^n = f\mu$ and $f\log f$ is integrable with respect to
$\mu$. We have the following result.

\begin{prop}\label{prop:Mproper}
  The functional $M_\omega$ is proper in the sense that there are
  constants $\delta, B > 0$ such that for all $u\in PSH_\omega(X)\cap L^\infty$ we have
  \[ \label{eq:Mprop1} M_{\omega}(u) > \delta \mathcal{J}_\omega(u) - B. \]
\end{prop}
\begin{proof}
  The case when $\lambda=1$ is well known, going back to
  Tian~\cite{Tian97}
  in the smooth
  setting, who proved a weaker version of properness. The properness
  in the form \eqref{eq:Mprop1} was shown by
  Phong-Song-Sturm-Weinkove~\cite{PSSW}.
  In the singular setting the result was shown in
  Darvas~\cite[Theorem 2.2]{DarvasMetric}. Note that we are assuming
  that $X$ has discrete automorphism group and admits a
  K\"ahler-Einstein metric. 

  The cases $\lambda=0, -1$ are much easier (see Tian~\cite{Tianbook} or
  Song-Weinkove~\cite[Theorem 1.2]{SW08} for a similar
  result). For this, note that $\mathcal{J}_\omega \geq 0$, 
  and so when $\lambda\leq 0$, we have
  \[ M_\omega(u) \geq \frac{1}{V} \int_{X^{reg}}
    \log\left(\frac{\omega_u^n}{\mu}\right) \omega_u^n. \]
  At the same time, using Tian~\cite{Tianalpha}, we know that there
  are $\alpha, C_1 > 0$ such that for all $u\in PSH_\omega(X)$ with
  $\sup_X u=0$ we have
  \[ \int_{Y} e^{-\alpha \pi^*u}\, \Omega < C_1,\]
  and so with $p^{-1} + q^{-1} =1$ (such that $F$ in \eqref{eq:FLp} is
  in $L^p$) we have
  \[ \int_{X^{reg}} e^{-\alpha q^{-1} u}\, d\mu &=
    \int_{\pi^{-1}(X^{reg})} e^{-\alpha q^{-1} \pi^*u}\, \pi^*\mu \\
    &= 
  \int_{\pi^{-1}(X^{reg})} e^{-\alpha q^{-1} \pi^*u }\, F
  \Omega \\
  &\leq \left(\int_{\pi^{-1}(X^{reg})} e^{-\alpha \pi^*u}\,
    \Omega\right)^{1/q} \left(\int_{\pi^{-1}(X^{reg})}
    F^p\Omega\right)^{1/p} \\
  &\leq C_2. 
\]
Using the convexity of the exponential function we then have, as in
\cite[Lemma 4.1]{SW08}, 
\[ \int_{X^{reg}} \log\left(\frac{\omega_u^n}{\mu}\right)\, \omega_u^n
  \geq \alpha q^{-1}\int_{X^{reg}} (-u)\, \omega_u^n - C_3, \]
for all $u\in PSH_\omega(X)$ with $\sup_X u =0$. As the same
time, if $\sup_X u =0$ and $u\in L^\infty$, then we have
\[ \int_{X^{reg}} (-u)\, \omega_u^n \geq \mathcal{J}_\omega(u). \]
To see this, note that
\[ \int_{X^{reg}} (-u)\, \omega_u^n &= \int_0^1 \frac{d}{dt}
  \int_{X^{reg}} (-tu)\, \omega_{tu}^n\, dt \\
  &= \int_0^1 \int_{X^{reg}} (-u)\, \omega_{tu}^n - n tu \ddbar
  u\wedge \omega_{tu}^{n-1}\, dt \\
  &\geq \int_0^1 n \int_{X^{reg}} u (\omega - \omega_{tu}) \wedge
  \omega_{tu}^{n-1}\, dt \\
  &= \int_0^1 \frac{d}{dt}\mathcal{J}_\omega(tu)\, dt =
  \mathcal{J}_\omega(u). \]
So combining the estimates above we obtain \eqref{eq:Mprop1}.
\end{proof}

Suppose that $\pi:Y \to X$ is a projective resolution such that the anticanonical
bundle $-K_Y$ is relatively nef. Let us write $E$ for the exceptional
divisor. The relatively nef assumption implies (see
Boucksom-Jonsson-Trusiani~\cite{BJT24}), that we have a smooth
volume form $\Omega$ on $Y$, whose Ricci form $\mathrm{Ric}(\Omega)$
satisfies
\[ \mathrm{Ric}(\Omega) \geq -C \pi^*\omega \]
for suitable $C > 0$. Let us fix a smooth K\"ahler metric $\eta_Y$ on
$Y$, with volume form $\Omega$, and we let $\eta_\epsilon =
\pi^*\omega + \epsilon\eta_Y$, which is a smooth K\"ahler metric on
$Y$. For any closed $(1,1)$-form $\alpha$ on $Y$, we define the
functional $\mathcal{J}_{\eta_\epsilon, \alpha}$ on $PSH_{\eta_\epsilon}(Y)\cap L^\infty$ by
letting $\mathcal{J}_{\eta_\epsilon, \alpha}(0)=0$ and
its variation
\[ \delta\mathcal{J}_{\eta_\epsilon, \alpha}(u) =
  n\int_Y \delta u \big(\alpha - c_\alpha\eta_{\epsilon,
    u}\big) \wedge \eta_{\epsilon, u}^{n-1}. \]
Here $c_\alpha$ is the constant determined by
\[ \int_Y  \big(\alpha - c_\alpha\eta_{\epsilon,
    u}\big) \wedge \eta_{\epsilon, u}^{n-1} = 0, \]
and $\eta_{\epsilon, u} = \eta_\epsilon + \ddbar u$.

We write
$\mathcal{J}_{\eta_\epsilon} = \mathcal{J}_{\eta_\epsilon,
  \eta_\epsilon}$, which is consistent with the earlier definition. 
The twisted Mabuchi K-energy in the class $[\eta_\epsilon]$ is defined, for
$u\in PSH_{\eta_\epsilon}(Y)\cap L^\infty$ by
\[ M_{\eta_\epsilon,s}(u) = \int_Y \log \left(\frac{\eta_{\epsilon,
        u}^n}{\Omega}\right)\, \eta_{\epsilon, u}^n +
    \mathcal{J}_{\eta_\epsilon, s\eta_\epsilon -
      \mathrm{Ric}(\Omega)}. \]
  Note that
  \[\label{eq:twistedMbigger}
    M_{\eta_\epsilon,s}(u)  \geq M_{\eta_\epsilon} :=
    M_{\eta_\epsilon, 0}\]
  for $s \geq 0$. 
The critical points of this functional are the twisted cscK metrics 
$\eta_{\epsilon, u}\in [\eta_\epsilon]$, satisfying
\[ R(\eta_{\epsilon, u}) - s\, \mathrm{tr}_{\eta_{\epsilon, u}} \eta_\epsilon=
  \mathrm{const.} \]

The following result uses our assumption that $-K_Y$ is
relatively nef.
\begin{lemma}\label{lem:KYnefproper}
  Assuming that $-K_Y$ is relatively nef,
  there is a constant $C_2 > 0$ such that $\mathcal{J}_{\eta_\epsilon,
    -\mathrm{Ric}(\Omega)} \geq - C_2 \mathcal{J}_{\eta_\epsilon}$ on
  $PSH_{\eta_\epsilon}(Y)\cap L^\infty$.
  In particular there are constants $s_0, \epsilon_0 > 0$ (depending on $(X,
  \omega_{KE})$) such that for $s \geq s_0$ and $\epsilon < \epsilon_0$ the twisted K-energy is
  proper: 
  \[ M_{\eta_\epsilon, s}(u) \geq \mathcal{J}_{\eta_\epsilon}(u), \]
  for all $u\in PSH_{\eta_\epsilon}(Y)\cap L^\infty$. 
\end{lemma}
\begin{proof}
   For $u\in PSH_{\eta_\epsilon}(Y)\cap L^\infty$ with $\sup_Yu = 0$,
  we have
  \[ -\mathcal{J}_{\eta_\epsilon, \mathrm{Ric}(\Omega)}(u) &=
    n \int_0^1
    \int_Y (-u) (\mathrm{Ric}(\Omega) - c \eta_{\epsilon, tu})\wedge
    \eta_{\epsilon, tu}^{n-1} \\
    &\geq - n\int_0^1
    \int_Y (-u) (C \pi^*\omega + c \eta_{\epsilon, tu})\wedge
    \eta_{\epsilon, tu}^{n-1} \\
    &\geq - C_1 n\int_0^1
    \int_Y (-u) (\eta_{\epsilon} + \eta_{\epsilon, tu})\wedge
    \eta_{\epsilon, tu}^{n-1} \\
    &\geq -C_2 J_{\eta_\epsilon}(u). 
  \]
  
  Note that since the entropy term is nonnegative, we have
  $M_{\eta_\epsilon, s} \geq \mathcal{J}_{\eta_\epsilon,
    s\eta_\epsilon - \mathrm{Ric}(\Omega)} $ and also
  \[ \mathcal{J}_{\eta_\epsilon, s\eta_\epsilon -
      \mathrm{Ric}(\Omega)} = s \mathcal{J}_{\eta_\epsilon,
      \eta_\epsilon} -
    \mathcal{J}_{\eta_\epsilon,\mathrm{Ric}(\Omega)}. \]
  It follows that for $s > C_2 + 1$,
  \[ M_{\eta_\epsilon, s}(u) \geq J_{\eta_\epsilon}(u). \]
\end{proof}

It follows from this result, using the work of
Chen-Cheng~\cite{ChenCheng}, that if $\epsilon < \epsilon_0$ and $s
> s_0$, then there exists a twisted cscK metric $\eta_{\epsilon, u}
\in [\eta_{\epsilon}]$ satisfying
\[ \label{eq:twistedcscK} R(\eta_{\epsilon,u}) - s\, \mathrm{tr}_{\eta_{\epsilon,u}}
  \eta_\epsilon = \mathrm{const.} \]
We will use a continuity method to construct twisted cscK metrics
in $[\eta_\epsilon]$ for sufficiently small $\epsilon$, that satisfy
\eqref{eq:twistedcscK} for $s\in [0,s_0]$, and so in particular we
obtain a cscK metric in $[\eta_\epsilon]$. For this we will need a
refinement of Chen-Cheng's estimates, which are uniform in the
degenerating cohomology classes $[\eta_\epsilon]$ as $\epsilon \to
0$. Such a refinement was shown by Zheng~\cite{Zheng20} who worked in
the more complicated setting of cscK metrics with cone
singularities. See also Pan-T\^o~\cite{PT24}. 

Note that in Zheng's work the cscK metrics are expressed relative
to metrics with a fixed volume form, rather than metrics of the form
$\eta_\epsilon$. Let us write $\tilde{\eta_\epsilon} \in
[\eta_\epsilon]$ for the metrics with $\tilde{\eta_\epsilon}^n = c_\epsilon
\Omega$ provided by Yau~\cite{Yau78}, where the $c_\epsilon$ are
bounded above and below uniformly. Note that we have
$\tilde{\eta_\epsilon} = \eta_\epsilon + \ddbar v_\epsilon$ with a
uniform bound on $\sup |v_\epsilon|$, independent of $\epsilon$, so it
does not matter whether we obtain $L^\infty$ bounds for potentials
relative to $\eta_\epsilon$ or relative to $\tilde{\eta_\epsilon}$.

In order to state the estimates in a form that we will use, we make
the following definition.

\begin{definition}
  Fix an exhaustion $K_1 \subset K_2 \subset \ldots \subset
  \pi^{-1}(X^{reg})$ of $\pi^{-1}(X^{reg})$ by compact sets.  Let
  $a_0, a_1, \ldots $ be a sequence of positive numbers, and $p > 1$. We say that a
  potential $u \in PSH_{\eta_\epsilon}(Y)$ is $\{p,a_j\}_{j\geq
    0}$-bounded, if we have
  \[ \left\Vert \frac{\eta_{\epsilon,
          u}^n}{\Omega}\right\Vert_{L^p(\Omega)} + 
      \sup_Y |u| \leq a_0, \qquad \sup_{K_j}
      \left|\log \frac{\eta_{\epsilon, u}^n}{\Omega}\right| + \Vert u\Vert_{C^4(K_j, \eta_Y)} \leq
    a_j. \]
  In other words such a potential is uniformly bounded globally, has
  volume form in $L^p$, is locally bounded in $C^4$, and its volume
  form is locally bounded above and below away from the exceptional
  divisor $E$. 
\end{definition}

We then have the following. 
\begin{prop}\label{prop:Zheng}
  Suppose that $\epsilon\in (0,1), s\in (0,s_0]$, and $\eta_{\epsilon,
    u} := \eta_\epsilon + \ddbar u$ satisfies the twisted cscK equation
  \[ R(\eta_{\epsilon, u}) - s \mathrm{tr}_{\eta_{\epsilon, u}}
    \eta_\epsilon = c_{s, \epsilon}, \]
  where $c_{s,\epsilon}$ is a constant determined by $s, \epsilon$
  through cohomological data. Assume that $\sup u = 0$.
  Let $\phi = \log |s_E|^2$, where $s_E$ is a section of
  $\mathcal{O}(E)$ vanishing along $E$, and we are using a smooth
  metric on $\mathcal{O}(E)$ to compute the norm.
  There are constants $C, a > 0$, $p > 1$, 
  depending on $Y, \eta_Y, \eta_0, s_0$, as well as on the entropy
  $\int_Y \log\left(\frac{\eta_{\epsilon, u}^n}{\Omega}\right)\,
    \eta_{\epsilon, u}^n$, but not on $\epsilon, s$,
  such that we have the following estimates:
  \begin{enumerate}
  \item \[\sup_Y\left( \log\frac{\eta_{\epsilon,u}^n}{\Omega} +
      a\phi\right) + \left\Vert \frac{\eta_{\epsilon,u}^n}{\Omega}
    \right\Vert_{L^p(\eta_Y)} + \sup_Y |u| < C, \]
  \item  \[ \inf_Y \left( \log\frac{\eta_{\epsilon,u}^n}{\Omega} -
        a\phi\right) > C, \]
  \item  \[ \Vert e^{a\phi}\mathrm{tr}_{\eta_Y} \eta_{\epsilon, u}\Vert_{L^p(\eta_Y)} <
      C. \]
  \end{enumerate}

  In particular there exist $p > 1$ and $a_j > 0$ such that $u$ is
  $\{p, a_j\}_{j\geq 0}$-bounded. 
\end{prop}
\begin{proof}
  The estimates (1) are shown in \cite[Proposition 5.12]{Zheng20}, the
  estimate (2) is in \cite[Proposition 5.15]{Zheng20}, and the estimate
  (3) is \cite[Proposition 5.18]{Zheng20}. Note that the $L^p$-bound
  on the trace of $\eta_{\epsilon, u}$ implies higher order estimates
  for $u$ on compact sets away from $E$. This leads to the $\{p,
  a_j\}$-boundedness of $u$. See also \cite[Theorem C]{PT24} for
  similar estimates. 
\end{proof}

Next we show that by Proposition~\ref{prop:Mproper}, the Mabuchi
energy $M_{\eta_\epsilon}$ is proper on $\{p, a_j\}$-bounded classes of potentials,
when $\epsilon$ is sufficiently small.

\begin{prop}\label{prop:smallepsproper}
  Given $p > 1$ and  a sequence $\{a_j\}_{j \geq 0}$, let $V \subset
  PSH_{\eta_\epsilon}(Y)$ denote the $\{p, a_j\}_{j\geq 0}$-bounded
  potentials. Then for sufficiently small $\epsilon$, depending on the
  $p, a_j$, the K-energy $M_{\eta_\epsilon}$ is proper on $V$ in the
  sense that
  \[ M_{\eta_\epsilon}(u) > \delta \mathcal{J}_{\eta_\epsilon}(u)- B_2,
    \text{ for all } u\in V. \]
  Here $\delta$ is the same constant as in Proposition~\ref{prop:Mproper},
  while $B_2$ is a constant depending on $(X, \omega)$ and
  $\Omega$, but not on the $p, a_j$. 
\end{prop}
\begin{proof}
  We argue by contradiction. Suppose that we have a sequence
  $\epsilon_i \to 0$, and $u_i \in PSH_{\eta_{\epsilon_i}}(Y)$ that
  are $\{p, a_j\}_{j\geq 0}$-bounded, such that
  \[ \label{eq:Muileq} M_{\eta_{\epsilon_i}}(u_i) \leq \delta
    \mathcal{J}_{\eta_\epsilon}(u) - B_2, \] 
  for $B_2$ to be determined below. 
  Up to choosing a subsequence we can assume that $u_i \to u_\infty$
  in $L^1$ and also in $C^{3,\alpha}$ on compact sets away from the
  exceptional divisor $E$. We have $u_\infty \in PSH_{\pi^*\omega}(Y)$,
  and we have an identification $PSH_{\pi^*\omega}(Y) =
  PSH_{\omega}(X)$.
  We will next show that in terms of $F$ in \eqref{eq:FLp} we have 
  \[ M_{\eta_{\epsilon_i}}(u_i) &\to M_{\omega}(u_\infty) + \int_Y \log
    F\, \eta_0^n, \\
     \mathcal{J}_{\eta_{\epsilon_i}}(u_i) &\to \mathcal{J}_\omega(u_\infty). \]
  Let us first consider the relevant entropy terms. Note that
  \[ \int_Y \log \left(\frac{\eta_{{\epsilon_i},
          u_i}^n}{\Omega}\right) \eta_{\epsilon_i, u_i}^n = \int_Y
    \log\left(\frac{ \eta^n_{\epsilon_i, u_i}}{\Omega}\right)
    \frac{\eta_{\epsilon_i, u_i}^n}{\Omega}\,
    \Omega. \]
  Our assumptions mean that the integrand has a
  uniform $L^p(\Omega)$-bound for some $p > 1$. Using this,
  and the $C^{3,\alpha}$-convergence $u_i\to u_\infty$ on compact sets
  away from $E$, it follows that
  \[ \int_Y \log \left(\frac{\eta_{{\epsilon_i},
          u_i}^n}{\Omega}\right) \eta_{\epsilon_i, u_i}^n &\to
    \int_Y \log \left(\frac{\eta_{0, u_\infty}^n}{\Omega}\right)
    \eta_{0, u_\infty}^n. \]
  Using \eqref{eq:FLp} we have
  \[ \int_Y \log \left(\frac{\eta_{0,u_\infty}^n}{\Omega}\right)
   \eta_{0, u_\infty}^n &= \int_{X^{reg}} \log
    \left(\frac{\omega_{u_\infty}^n}{\mu}\right) 
    \omega_{u_\infty}^n + \int_Y \log F\,
    \eta_{0,u_\infty}^n. \]
  The last term can be computed by writing
  \[ \int_Y \log F\,
    \eta_{0,u_\infty}^n &= \int_Y \log F\, \eta_0^n + \int_0^1
    \frac{d}{dt} \int_Y \log F\, \eta_{0, tu_\infty}^n\, dt \\
    &= \int_Y \log F\, \eta_0^n + \int_0^1 n \int_Y u_\infty
    \ddbar\log F\, \wedge \eta_{0,t u_\infty}^{n-1}\,dt \\
    &= \int_Y \log F\, \eta_0^n + \int_0^1 n \int_Y u_\infty
    (\mathrm{Ric}(\Omega) - \mathrm{Ric}(\pi^*\mu))\wedge
    \eta_{0,u_\infty}^{n-1}\,dt \\
    &= \int_Y \log F\, \eta_0^n + \mathcal{J}_{\eta_0,
      \mathrm{Ric}(\Omega)}(u_\infty) -\lambda \mathcal{J}_{\omega}(u_\infty).
  \]
  For the last step note that $\eta_0$ vanishes along $E$, so although
  $\mathrm{Ric}(\pi^*\mu)$ has current contributions along $E$, the
  only part that survives in the integral is $\mathrm{Ric}(\mu) =
  \lambda\omega$ on $X$.  In conclusion we have that
  \[ \label{eq:entropyconverge}\int_Y \log \left(\frac{\eta_{{\epsilon_i},
          u_i}^n}{\Omega}\right) \eta_{\epsilon_i, u_i}^n &\to
    \int_{X^{reg}} \log\left(\frac{\omega^n_{u_\infty}}{\mu}\right)\,
    \omega_{u_\infty}^n + \int_Y \log F\, \eta_0^n \\ &\qquad +
    \mathcal{J}_{\eta_0, \mathrm{Ric}(\Omega)}(u_\infty) - \lambda
    \mathcal{J}_\omega(u_\infty). \]

  Next we consider the $\mathcal{J}$-functional terms. Consider a
  general smooth, closed (1,1)-form $\alpha$ on $Y$. We claim that we have
  $\mathcal{J}_{\eta_{\epsilon_i}, \alpha}(u_i) \to
  \mathcal{J}_{\eta_0, \alpha}(u_\infty)$. Using the variational
  definition of $\mathcal{J}$,
  the local $C^{3,\alpha}$-convergence, and the uniform
  $L^\infty$-bound for the $u_i$,  it is enough to show
  that for every $\kappa > 0$ there is a compact set $K\subset
  Y\setminus E$, such that
  \[ \label{eq:YKint} \int_{Y\setminus K} \eta_1 \wedge \eta_{\epsilon_i, u_i}^{n-1} +
    \int_{Y\setminus K} \eta_{\epsilon_i, u_i}^n < \kappa, \text{ for
      all } i. \]
  To see this, let $h= -\log |s_E|^2$, where $s_E$ is a section of the
  line bundle $\mathcal{O}(E)$ over $Y$ vanishing along the exceptional divisor
  $E$, and we use a smooth metric on $\mathcal{O}(E)$. We have
  \[ \ddbar h = \chi - [E], \]
  where $\chi$ is a smooth form on $Y$. We can assume that $h\geq 0$,
  and note that $h\to\infty$ along $E$. We show by induction that for each $k=0,
  \ldots, n$ there is a constant $C_k > 0$, independent of $i$, such that 
  \[\label{eq:hintbound} \int_Y h \eta_1^{n-k}\wedge \eta_{\epsilon_i, u_i}^k \leq C_k. \]
  For $k=0$ this is clear since $h$ has logarithmic
  singularities. Suppose that the bound has been established for a
  value of $k$. Then
  \[ \int_Y h \eta_1^{n-k-1}\wedge \eta_{\epsilon_i, u_i}^{k+1} &=
    \int_Y h \eta_1^{n-k-1} \wedge (\eta_{\epsilon_i} + \ddbar
    u_i)\wedge \eta_{\epsilon_i, u_i}^k \\
    &= \int_Y h \eta_{1}^{n-k-1}\wedge \eta_{\epsilon_i}\wedge \eta_{\epsilon_i, u_i}^k + \int
    u_i\ddbar h \wedge \eta_1^{n-k-1}\wedge \eta_{\epsilon_i, u_i}^k
    \\
    &\leq \int_Y h \eta_{1}^{n-k}\wedge \eta_{\epsilon_i, u_i}^k +
    \int_Y u_i \chi\wedge \eta_1^{n-k-1}\wedge \eta_{\epsilon_i,
      u_i}^k - \int_E u_i \eta_1^{n-k-1}\wedge \eta_{\epsilon_i, u_i}^k \\
    &\leq C_k(1+C) - \int_E u_i \eta_1^{n-k-1}\wedge \eta_{\epsilon_i,
      u_i}^k \\
    &\leq C_k(1+ C) + C', 
  \]
  where $C, C'$ depend on $\chi$ and the uniform $L^\infty$ bound for
  $u_i$.
  
  Since $h\to\infty$ along $E$, it follows from \eqref{eq:hintbound}
  that for any $\kappa > 0$ we can find a compact set $K\subset
  Y\setminus E$ such that \eqref{eq:YKint} holds. 
  It follows that
  \[ \mathcal{J}_{\eta_{\epsilon_i, -\mathrm{Ric}(\Omega)}}(u_i) \to
    \mathcal{J}_{\eta_{0, -\mathrm{Ric}(\Omega)}}(u_\infty), \]
  and also
  \[ \mathcal{J}_{\eta_{\epsilon_i}}(u_i) \to
    \mathcal{J}_{\omega}(u_\infty). \]

  From this, together with \eqref{eq:entropyconverge}, we have
  \[ M_{\eta_{\epsilon_i}}(u_i) &\to \int_{X^{reg}} \log\left(
      \frac{\omega_{u_\infty}^n}{\mu}\right)\, \omega_{u_\infty}^n
    -\lambda \mathcal{J}_\omega(u_\infty) + \int_Y \log F\, \eta_0^n
    \\
    &= M_{\omega}(u_\infty) + \int_Y \log F\, \eta_0^n. \]
  From \eqref{eq:Muileq} we therefore get
  \[ M_{\omega}(u_\infty) + \int_Y \log F\, \eta_0^n \leq
    \delta\mathcal{J}_\omega(u_\infty) - B_2. \]
  Choosing $B_2 = B - \int_Y \log F\, \eta_0^n$ for the $B$ in 
  Proposition~\ref{prop:Mproper}, we get a contradiction. 
\end{proof}

We are now ready to combine the different ingredients to prove the
main result of this section.
\begin{proof}[Proof of Theorem~\ref{thm:cscKapprox}]
  We will choose suitable $p > 0, a_j > 0$ shortly. By
  Proposition~\ref{prop:smallepsproper}, for a given $p, a_j$ we have
  some $\epsilon_1 > 0$ such that once $\epsilon < \epsilon_1$ and
  for any $s\geq 0$, we have
  \[ M_{\eta_\epsilon, s}(u) \geq M_{\eta_\epsilon}(u) > \delta \mathcal{J}_{\eta_\epsilon}(u)-
    B_2, \]
  for $\{p. a_j\}$-bounded potentials $u$. Recall that $\delta, B_2$
  do not depend on $\{p, a_j\}$. For small $\kappa > 0$ we
  have
  \[ M_{\eta_\epsilon, s}(u) &\geq \kappa \int_Y
    \left(\frac{\eta_{\epsilon, u}^n}{\Omega}\right)\, \eta_{\epsilon,
      u}^n + \kappa \mathcal{J}_{\eta_\epsilon,
      s\eta_\epsilon-\mathrm{Ric}(\Omega)}(u) + (1-\kappa)\delta
    \mathcal{J}_{\eta_\epsilon}(u) - (1-\kappa)B_2 \\
  &= \kappa \int_Y
    \left(\frac{\eta_{\epsilon, u}^n}{\Omega}\right)\, \eta_{\epsilon,
      u}^n + (\kappa s + (1-\kappa)\delta)
    \mathcal{J}_{\eta_\epsilon}(u) + \kappa
    \mathcal{J}_{\eta_\epsilon, -\mathrm{Ric}(\Omega)}(u) -
    (1-\kappa)B_2. 
  \]
  If $\kappa$ is chosen sufficiently small (depending on $\delta$),
  then by Lemma~\ref{lem:KYnefproper} we find that
  \[ \label{eq:twistedMproper10} M_{\eta_\epsilon, s}(u) &\geq \kappa
    \int_Y \log
    \left(\frac{\eta_{\epsilon, u}^n}{\Omega}\right)\, \eta_{\epsilon,
      u}^n - B_2. \]
  We also have
  \[ M_{\eta_\epsilon, s}(0) = \int_Y \log
    \left(\frac{\eta_{\epsilon}^n}{\Omega}\right)\, \eta_{\epsilon}^n
    < C_3, \]
  for a constant $C_3 > 0$ independent of $\epsilon$. Since twisted
  cscK metrics minimize the twisted Mabuchi K-energy, it follows that
  if $\eta_{\epsilon, u} \in [\eta_\epsilon]$ is a twisted cscK
  metric, then we have $M_{\eta_\epsilon, s}(u) < C_3$. From
  \eqref{eq:twistedMproper10} we get
  \[ \label{eq:entbound10} \int_Y \log \left(\frac{\eta_{\epsilon, u}^n}{\Omega}\right)\, \eta_{\epsilon,
      u}^n \leq \kappa^{-1}(C_3 + B_2), \]
  and in particular the entropy of $\eta_{\epsilon, u}$ is bounded
  independently of $\epsilon$. We  apply
  Proposition~\ref{prop:Zheng}. As long as $s\leq s_0$, for the $s_0$
  determined by Lemma~\ref{lem:KYnefproper}, we find that if
  $\eta_{\epsilon, u} = \eta_\epsilon + \ddbar u$ is a solution of the
  twisted cscK equation
  \[\label{eq:twistedcscK10} R(\eta_{\epsilon, u}) - s\,\mathrm{tr}_{\eta_{\epsilon, u}}
    \eta_\epsilon = \mathrm{const.}, \]
  then $u$ is $\{p, a_j\}$-bounded, for suitable $p, a_j$, determined
  by $s_0$ and the entropy bound \eqref{eq:entbound10}. From now we
  fix this choice of  $p, a_j$.

  We can now use a continuity method to show that if $\epsilon <
  \epsilon_1$, for the $\epsilon_1$ determined by $\{p, a_j\}$,  for all $s\in
  [0,s_0]$ we can solve the twisted cscK equation
  \eqref{eq:twistedcscK10}.  To see this, let us fix $\epsilon < \epsilon_1$,
  and set
  \[ S = \{ s\in [0,s_0]\, :\, 
    \text{ the equation \eqref{eq:twistedcscK10} has a solution}\}.  \]
  We have $s_0\in S$, and it follows from the implicit function theorem
  that $S$ is open. To see that it is closed, note that the twisted
  cscK metrics   for $s\in S$
  automatically satisfy the entropy bound \eqref{eq:entbound10}.
  Using the main estimates of Chen-Cheng~\cite{ChenCheng}, we find
  that the potentials of the corresponding twisted cscK metrics
  satisfy a priori $C^k$-estimates, and the metrics are bounded below
  uniformly (these estimates depend on $\epsilon$, but now $\epsilon$
  is fixed). It follows that $S$ is closed.

  It follows that for sufficiently small $\epsilon > 0$ the classes
  $[\eta_\epsilon]$ on $Y$ admit cscK metrics. The estimates required
  by Definition~\ref{defn:cscKapprox} follow from
  Proposition~\ref{prop:Zheng}. 
\end{proof}

\begin{remark}
  To conclude this section we give an example where the assumption that
$-K_Y$ is relatively nef is satisfied. Let $M$ be a smooth Fano
manifold, and suppose that $P$ is a line bunde over $M$ such that $P^r
= -K_M$ for some $r > 0$. We let $V$ denote the total space of
$P^{-1}$, with the zero section blown down to a point $o$. Suppose
that $X$ has one isolated
singularity $p$, and a neighborhood of $p$ is
isomorphic to the neighborhood of $o\in V$. In this case we can
consider a resolution $\pi : Y\to X$, obtained by blowing up the singular
point. Then
\[ K_Y = \pi^*K_X + rE, \]
where the exceptional divisor $E$ isomorphic to $M$, and is in
particular irreducible. It follows that
in this case $-K_Y$ is relatively nef (in fact relatively ample). Note
that this family of examples does not fit into the framework of
admissible singularities studied by Li-Tian-Wang~\cite{LTW21}.
\end{remark}

\section{Partial $C^0$-estimate}\label{sec:bounded}
An important result of Donaldson-Sun~\cite{DS1} is the partial
$C^0$-estimate for smooth
K\"ahler-Einstein manifolds, conjectured by Tian~\cite{Tian90}. 
More precisely, suppose that $(X, \omega_{KE})$ is
a smooth K\"ahler-Einstein manifold, with $\omega_{KE}\in c_1(L)$ for
an ample line bundle, and such that for some constant $D > 0$ we have
\begin{enumerate}
  \item non-collapsing: $\mathrm{vol}\, B_{\omega_{KE}}(p, 1) > D^{-1}$
    for a basepoint $p\in X$,
  \item bounded volume: $\mathrm{vol} (X, \omega_{KE}) < D$,
  \item bounded Ricci curvature: $\mathrm{Ric}(\omega_{KE}) = \lambda
    \omega_{KE}$ for $|\lambda| < D$.
  \end{enumerate}
For any integer $k > 0$ the density of states function $\rho_{k,
  \omega_{KE}}$ is defined by
\[ \rho_{k, \omega_{KE}}(x) = \sum_j |s_j|^2(x), \]
where the $s_j$ form an $L^2$-orthonormal basis of $H^0(X, L^k)$ in
terms of the metric induced by $k\omega_{KE}$. 
Then, by Donaldson-Sun~\cite{DS1}, there is a power $k_0 = k_0(n, D)$, and $b=b(n,D) > 0$,
depending on the dimension and 
the constant $D$, such that $\rho_{k_0, \omega_{KE}} > b$. In this
section we show the following extension of this
result to singular K\"ahler-Einstein
spaces that admit good cscK approximations.

\begin{thm}\label{thm:bounded1}
  Given $n, D > 0$ there are constants $k_0(n,D), b(n,D) > 0$ with the following
  property. Suppose that $(X,\omega_{KE})$ is a singular
  K\"ahler-Einstein variety of dimension $n$, such that $\omega_{KE}\in c_1(L)$ for a
  line bundle $L$. Assume that  $(X, \omega_{KE})$ can be approximated by cscK
  metrics, and in addition the conditions (1), (2),
  (3) above hold. Then the corresponding density of states function
  satisfies $\rho_{k, \omega_{KE}} > b$. 
\end{thm}

The proof of the result follows the same strategy as
Donaldson-Sun~\cite{DS1}, arguing by contradiction. We suppose that
the sequence $(X_i, \omega_{KE,   i})$ satisfies the bounds (1)--(3),
but no fixed power $L_i^k$ of the corresponding line bundles is very
ample. The corresponding metric completions $\hat{X}_i$ are
non-collapsed RCD spaces by Proposition~\ref{prop:RCD},
and we can pass to the Gromov-Hausdorff
limit $\hat{X}_\infty$ along a subsequence. We would then like to use
the structure of the tangent cones of $\hat{X}_\infty$ to construct
suitable holomorphic sections of a suitable power $L_i^k$ for large
$i$, leading to a contradiction.

The difficulty in executing this strategy is that we do not have
good control of the convergence of $\hat{X}_i$ to $\hat{X}_\infty$ on
the regular set of $\hat{X}_\infty$, because in
Corollary~\ref{cor:regset} the constant $\epsilon$ depends on
the singular K\"ahler-Einstein space $X$ that we are considering. As
such it is a priori possible that the singular set of
$\hat{X}_\infty$, consisting of points where the tangent cone is not
given by $\mathbb{R}^{2n}$, is dense. In order to rule this out, we
prove the following. Note that recently this result was shown in the
more general algebraic setting by Xu-Zhuang~\cite{XZ24} (see also
Liu-Xu~\cite{LX19} for the three dimensional case).

\begin{thm}\label{thm:gap} 
  There is an $\epsilon > 0$, depending only on the dimension $n$, with
  the following property. Suppose that $\hat{X}$ is the
  metric completion of a singular K\"ahler-Einstein space as in
  Theorem~\ref{thm:homeomorphic}, i.e. one that can be 
  approximated by cscK metrics. Let $(\hat{X}_p, o)$ be a tangent cone
  of $\hat{X}$, such
  that $\hat{X}_p \not= \mathbb{R}^{2n}$. Then
  \[ \mathrm{vol} B(o, 1) < \omega_{2n} - \epsilon, \]
  where $\omega_{2n}$ is the volume of the Euclidean unit ball in
  $\mathbb{R}^{2n}$. 
\end{thm}
\begin{proof}
  We will argue by contradiction. If the stated result is not true,
  then we can find a sequence $\hat{X}_i$, and a sequence of singular
  points $p_i\in \hat{X}_i$ with tangent cones $V_{p_i}$ such that
  $V_{p_i}\to \mathbb{R}^{2n}$ in the pointed Gromov-Hausdorff sense.

  We will prove a more general statement about almost smooth metric
  measure spaces in the sense of Definition~\ref{defn:almostsmooth},
  of any dimension, which satisfy the following conditions.

  \begin{definition} We say that an almost smooth metric measure space $V$
    satisfies Condition ($\ast$) if the following conditions hold: 
  \begin{enumerate}
  \item For some $\epsilon > 0$ (possibly depending on $V$), the
    $\epsilon$-regular set $\mathcal{R}_\epsilon \subset V$, defined
    by \eqref{eq:epsreg}, can be chosen to be the set $\Omega$ in
    Definition~\ref{defn:almostsmooth}.
   \item The Riemannian metric on $\Omega$ is Ricci flat. 
  \item If a tangent cone $V'$ of $V$ is of the form
    $C(S^1_\gamma)\times \mathbb{R}^{2n-2}$, then $V' =
    \mathbb{R}^{2n}$. 
  \end{enumerate}
\end{definition}

  Note that by Propositions~\ref{prop:almostreg31} and
  \ref{prop:nocodim2}, the (iterated) tangent cones of the spaces $\hat{X}_i$
  satisfy Condition ($\ast$). Moreover, if a space $V = W \times
  \mathbb{R}^j$ satisfies Condition ($\ast$), then so does $W$, and so
  do the tangent cones of $V$. 

  We argue by induction on the dimension to
  show that if a sequence of $k$-dimensional cones $V_j$ satisfies Condition ($\ast$), and $V_j
  \to \mathbb{R}^k$ in the pointed Gromov-Hausdorff sense, then $V_j =
  \mathbb{R}^k$ for sufficiently large $i$. For $k=2$ this follows directly
  from Condition ($\ast$).

  Assuming $k > 2$, 
  suppose first that for all sufficiently large $j$ the cones $V_j$
  have smooth link (i.e. the singular set consists of only the
  vertex). In this case $V_j = C(Y_j)$, where the $(Y_j, h_j)$ are
  $(k-1)$-dimensional smooth Einstein manifolds satisfying
  $\mathrm{Ric}(h_j) = (k-2)h_j$. Moreover the $(Y_j, h_j)$ converge
  in the Gromov-Hausdorff sense to the unit $(k-1)$-sphere.
  As long as $k-1 > 1$, it follows that for sufficiently large $j$ we have
  $\mathrm{vol}(Y_j, h_j) = \mathrm{vol}(S^{k-1}, g_{S^{k-1}})$, using
  that Einstein metrics are critical points of the Einstein-Hilbert
  action. The Bishop-Gromov comparison theorem then implies that in
  fact $(Y_j, h_j)$ is isometric to the unit $(k-1)$-sphere for
  sufficiently large $j$, so that $V_j = \mathbb{R}^{k}$. If $k-1=1$,
  then $V_j$ is a cone over a circle, so by 
  Condition ($\ast$) we have $V_j = \mathbb{R}^2$. Either way we have
  a contradiction. 

  We can therefore assume, up to choosing a subsequence, that the $V_j$
  all have singularities $q_j$ away from the vertex. By taking tangent
  cones at the $q_j$, we obtain a new sequence of cones, $V_j'$, which
  still satisfy the Condition ($\ast$), they converge to
  $\mathbb{R}^k$, and they all split off an isometric factor of
  $\mathbb{R}$, i.e. $V_j' = W_j \times \mathbb{R}$. The cones $W_j$
  are then $k-1$ dimensional, they also satisfy Condition ($\ast$),
  and $W_j \to \mathbb{R}^{k-1}$. We can then apply the inductive
  hypothesis. It follows that $W_j = \mathbb{R}^{k-1}$ for large $j$,
  so $V_j'= \mathbb{R}^k$, contradicting that the $q_j$ are singular
  points. 
\end{proof}

Given this result, we can follow the argument of
Donaldson-Sun~\cite{DS1} to prove Theorem~\ref{thm:bounded1}.
\begin{proof}[Proof of Theorem~\ref{thm:bounded1}]
  We argue by contradiction. Suppose that there are singular
  K\"ahler-Einstein spaces $(X_i, \omega_{KE, i})$, that can be
  approximated by cscK metrics, with
  $\omega_{KE,i}\in c_1(L_i)$, satisfying the
  conditions (1)--(3) before the statement of
  Theorem~\ref{thm:bounded1}, but such that there is no fixed power
  $L_i^k$ of  the line bundles $L_i$ whose density of states functions
  are bounded away from zero uniformly.
  Up to choosing a subsequence, we can assume that the
  corresponding RCD spaces $\hat{X}_i$ converge to $\hat{X}_\infty$ in
  the Gromov-Hausdorff sense. Theorem~\ref{thm:gap} implies
  that for some $\epsilon > 0$, the $\epsilon$-regular subset of
  $\hat{X}_\infty$ coincides with the regular set $\mathcal{R}\subset
  \hat{X}_\infty$ (given by the points with tangent cone
  $\mathbb{R}^{2n}$). Therefore the set $\mathcal{R}$ is open, and by
  Theorem~\ref{thm:gap} together with
  Proposition~\ref{prop:almostreg31}, it follows that 
  the convergence $\hat{X}_i\to\hat{X}_\infty$
  is locally smooth on $\mathcal{R}$. In addition, using the argument
  in Proposition~\ref{prop:nocodim2}, we know that no iterated tangent
  cone of $\hat{X}_\infty$ is given by
  $C(S^1_\gamma)\times\mathbb{R}^{2n-2}$ with $\gamma < 2\pi$. This
  means that we are in essentially the same setting as
  Donaldson-Sun~\cite{DS1}, and can closely follow their arguments
  to show that there is a  $k_0 > 0$, such that the density of
  states functions of the sections of $L_i^{k_0}$ are bounded away from
  zero for all sufficiently large  $i$. 
\end{proof}

\end{document}